\documentclass[reqno]{amsart}
\usepackage{graphicx}
\usepackage{amstext}
\usepackage{amssymb}
\usepackage{amsmath}
\newtheorem{theorem}{Theorem}[section]
\newtheorem{lemma}[theorem]{Lemma}
\newtheorem{corollary}[theorem]{Corollary}
\newtheorem{proposition}[theorem]{Proposition}
\theoremstyle{definition}
\newtheorem{definition}[theorem]{Definition}

\theoremstyle{remark}
\newtheorem{remark}[theorem]{Remark}
\numberwithin{equation}{section}

\begin{document}

\title{STABLE SELF-SIMILAR BLOW-UP DYNAMICS FOR SLIGHTLY $L^2$-SUPERCRITICAL GENERALIZED KDV EQUATIONS}\def\rightmark{BLOW-UP FOR SUPERCRITICAL GKDV EQUATIONS}

\author{Yang LAN}

\address{Laboratoire de Mathematiques D'Orsay, Universite Paris-Sud, Orsay, France}

\email{yang.lan@math.u-psud.fr}


\begin{abstract}
In this paper we consider the slightly $L^2$-supercritical gKdV equations $\partial_t u+(u_{xx}+u|u|^{p-1})_x=0$, with the nonlinearity $5<p<5+\varepsilon$ and $0<\varepsilon\ll 1$ . We will prove the existence and stability of a blow-up dynamics with self-similar blow-up rate in the energy space $H^1$ and give a specific description of the formation of the singularity near the blow-up time.
\end{abstract}

\maketitle

\section{Introduction}

\subsection{Setting of the problem}
We consider the following gKdV equations:
\begin{equation}\label{CP}\begin{cases}
\partial_t u +(u_{xx}+u|u|^{p-1})_x=0, \quad (t,x)\in[0,T)\times\mathbb{R},\\
u(0,x)=u_0(x)\in H^1(\mathbb{R}),
\end{cases}
\end{equation}
with $1\leq p<+\infty$.

From the result of C. E. Kenig, G. Ponce and L. Vega \cite{KPV} and N. Strunk \cite{S}, \eqref{CP}  is locally well-posed in $H^1$ and thus for all $ u_0\in H^1$, there exists a maximal lifetime $0<T\leq+\infty$ and a unique solution $u(t,x)\in C([0,T), H^1(\mathbb{R}))$ to \eqref{CP}. Besides, we have the blow-up criterion: either $T=+\infty$ or $T<+\infty$ and $\lim_{t\rightarrow T}\|u_x(t)\|_{L^2}=+\infty$.

\eqref{CP} admits 2 conservation laws, i.e. the mass and energy:
\begin{gather*}
M(u(t))=\int |u(t,x)|^2 dx=M(u(0)),\\
E(u(t))=\frac{1}{2}\int |u(t,x)|^2 dx-\frac{1}{p+1}\int |u(t,x)|^{p+1} dx=E(u(0)).
\end{gather*}

For all $ \lambda>0$, $u_{\lambda}(t,x)=\lambda^{\frac{2}{p-1}}u(\lambda^3t,\lambda x)$ is also a solution which leaves the Sobolev space $\dot{H}^{\sigma_c}$ invariant with the index:
\begin{equation}
\sigma_c=\frac{1}{2}-\frac{2}{p-1}.
\end{equation}

We introduce the  ground state $\mathcal{Q}_p$, which is the unique radial nonnegative function with exponential decay at infinity to the following equation:
\begin{equation}
\mathcal{Q}_p''-\mathcal{Q}_p+\mathcal{Q}_p|\mathcal{Q}_p|^{p-1}=0.
\end{equation}
$\mathcal{Q}_p$ plays a distinguished role in the analysis. It provides a family of travelling wave solutions:
$$u(t,x)=\lambda^{\frac{2}{p-1}}\mathcal{Q}_p(\lambda(x-\lambda^2 t-x_0)),\quad (\lambda,x_0)\in\mathbb{R}_{+}^*\times\mathbb{R}.$$

For $p<5$ or equivalently $\sigma_c<0$, \eqref{CP} is subcritical in $L^2$. The mass and energy conservation laws imply that the solution is always global and bounded in $H^1$. So a necessary condition for the occurrence of blow-up is $p\geq 5$. For $p=5$, the blow up dynamics have been studied in a series of papers of Y. Martel, F. Merle and P. Rapha\"{e}l in \cite{M1,MM1,MM2,MM4,MM5,MM3,MMR1,MMR2,MMR3}.

\subsection{On  the supercritical problem}
Let us first consider the focusing $L^2$ supercritical NLS equations:
$$\begin{cases}
i\partial_t u +\Delta u+|u|^{p-1}u=0, \quad (t,x)\in[0,T)\times\mathbb{R}^d,\\
u(0,x)=u_0(x)\in H^1(\mathbb{R}^d),
\end{cases}$$
with nonlinearity $p>1+\frac{4}{d}$. In \cite{MRS1}, F. Merle, P. Rapha\"{e}l and J. Szeftel  show that for $d\geq 2$, there are radial solutions which blow up on an asymptotic blow-up sphere instead of a blow-up point. And in \cite{MRS}, F. Merle, P. Rapha\"{e}l and J. Szeftel construct a stable self-similar blow-up dynamics for slightly $L^2$-supercritical nonlinearity, with nonradial initial data in low dimension (i.e. $d\leq 5$).

Now let us return to the gKdV equations. In this paper we consider the slightly supercritical case:
$$5<p<5+\varepsilon,\quad 0<\varepsilon\ll1.$$

The explicit description of blow-up dynamics for supercritical gKdV equations is mostly open. But numerical simulation of D. B. Dix and W. R. McKinney \cite{DM} suggests that there are self-similar blow-up solutions to supercritical gKdV equations%
\footnote{We know from \cite{MM3} that there are no self-similar blow-up solutions for the $L^2$-critical gKdV equation.}%
. We can expect a similar result to the slightly supercritical Schr\"{o}dinger equations, i.e. \cite{MRS}. More precisely, we expect a blow-up solution of the following form:
$$u(t,x)\sim\frac{1}{\lambda(t)^{\frac{2}{p-1}}}P(\frac{x}{\lambda(t)}),\quad \lambda(t)\sim \sqrt[3]{T-t}.$$
But here the delicate issue is that the profile $P$ seems not to be provided by the ground state $\mathcal{Q}_p$.  If we explicitly let:
$$u(t,x)=\frac{1}{\lambda(t)^{\frac{2}{p-1}}}Q_b(\frac{x}{\lambda(t)}),\quad \lambda(t)=\sqrt[3]{3b(T-t)},\quad b>0.$$
Then $u$ solves \eqref{CP} if and only if $Q_b(y)$ solves the following ODE%
\footnote{See the definition of \lq\lq$\Lambda$" in Section 1.4.}%
:
\begin{equation}\label{PF}
b\Lambda Q_b+(Q_b''-Q_b+Q_b|Q_b|^{p-1})'=0.
\end{equation}

The exact solutions of \eqref{PF} have been exhabited by H. Koch \cite{K}, for the slightly supercritical nonlinearity $5<p<5+\varepsilon$, $0<\varepsilon\ll 1$. It is related to an eigenvalue problem, i.e. for all $5<p<5+\varepsilon$, there exists an unique $b=b(p)>0$, such that a unique smooth solution $Q_b$ to \eqref{PF} with zero energy is found. Moreover $Q_{b}$ belongs to
$\dot{H}^1\cap L^{p+1}$, but always misses the invariant Sobolev space $\dot{H}^{\sigma_c}$ (hence $Q_b\notin L^2$) due to a slowly decaying tail at the infinity:
$$Q_b(y)\sim\frac{1}{|y|^{\frac{1}{2}-\sigma_c}}.$$

Despite the slowly decaying tail, we can choose a suitable cut-off of $Q_b$ as an approximation, such that it is bounded in $L^2$ with exponential decay on the right. We claim that the approximate self-similar profile generates a stable self-similar blow-up dynamics for the time dependent problems.

\subsection{Statement of the result}
\begin{theorem}[Existence and stability of a self-similar blow-up dynamics]\label{MT}
There exists a $p^*>5$ such that for all $p\in(5,p^*)$, there exist constants $\delta(p)>0$ and $b^*(p)>0$ with
\begin{gather}
\lim_{p\rightarrow5}\delta(p)=0\\
0<c_0(p-5)\leq b^*(p)\leq C_0(p-5)
\end{gather}
and a nonempty open subset $\mathcal{O}_p$ in $H^1$ such that the following holds. If $u_0\in \mathcal{O}_p$, then the corresponding solution to \eqref{CP} blows up in finite time $0<T<+\infty$, with the following dynamics : there exist geometrical parameters $(\lambda(t),x(t))\in\mathbb{R}_{+}^*\times\mathbb{R} $ and an error term $\varepsilon(t)$ such that:
\begin{equation}\label{11}
u(t,x)=\frac{1}{\lambda(t)^{\frac{2}{p-1}}}\big{[}\mathcal{Q}_p+\varepsilon(t)\big{]}\bigg{(}\frac{x-x(t)}{\lambda(t)}\bigg{)}
\end{equation}
with
\begin{equation}\label{12}
\|\varepsilon_{y}(t)\|_{L^2}\leq \delta(p).
\end{equation}
Moreover, we have:
\begin{enumerate}
\item The blow-up point converges at the blow-up time:
\begin{equation}\label{13}
x(t)\rightarrow x(T) \text{ as } t\rightarrow T,
\end{equation}
\item The blow-up speed is self-similar:
\begin{equation}\label{14}
\forall t\in[0,T),\quad (1-\delta(p))\sqrt[3]{3b^*(p)}\leq\frac{\lambda(t)}{\sqrt[3]{T-t}}\leq(1+\delta(p))\sqrt[3]{3b^*(p)}.
\end{equation}
\item The following convergence holds:
\begin{equation}\label{15}
\forall q\in[2,\frac{2}{1-2\sigma_c}), \quad u(t)\rightarrow u^*\text{ in $L^q$ as $t\rightarrow T$}.
\end{equation}
\item The asymptotic profile $u^*$ displays the following singular behavior:
\begin{equation}\label{16}
\big{(}1-\delta(p)\big{)}\int \mathcal{Q}^2_p\leq \frac{1}{R^{2\sigma_c}}\int_{|x-x(T)|<R}|u^*|^2\leq \big{(}1+\delta(p)\big{)}\int \mathcal{Q}^2_p.
\end{equation}
for $R$ small enough.
In particular, we have for all $ q\geq \frac{2}{1-2\sigma_c}$:
$$u^*\notin L^q.$$
\end{enumerate}
\end{theorem}
\begin{remark}
Here the meaning of $q_c=\frac{2}{1-2\sigma_c}$ is given by the following Sobolev embedding:
$$\dot{H}^{\sigma_c}\hookrightarrow L^{q_c}.$$
That is, the asymptotic profile $u^*$ is not in the critical space $\dot{H}^{\sigma_c}$, and the strong convergence \eqref{15} only exists in subcritical Lebesque spaces.
\end{remark}
\begin{remark}
It is easy to see from the $L^2$ conservation law that $\int|u^*|^2=\int|u_0|^2$.
\end{remark}
\begin{remark}
Theorem \ref{MT} is the first construction of blow-up solutions to the supercritical gKdV equations with initial data in $H^1$. This is a stable blow-up dynamics instead of a single blow-up solution. So it is not like the self-similar solution constructed by H. Koch in \cite{K}, though the construction in this paper relies deeply on H. Koch's work.
\end{remark}

\subsection{Notation}
We first introduce the associated scaling generators:
\begin{equation}
\Lambda f=\frac{2}{p-1}f+yf'.
\end{equation}
We denote the $L^2$ scalar product by:
\begin{equation}
(f,g)=\int_{\mathbb{R}}f(x)g(x)dx
\end{equation}
and observe the integration by parts:
\begin{equation}
(\Lambda f,g)=-(f,\Lambda g+2\sigma_c g).
\end{equation}
Then we let $\mathcal{Q}_p$ be the ground state. For $p=5$, we simply write $\mathcal{Q}_p$ as $\mathcal{Q}$. We introduce the linearized operators at $\mathcal{Q}_p$:
\begin{equation}
Lf=-f''+f-p\mathcal{Q}_p^{p-1}f.
\end{equation}
A standard computation leads to:
\begin{equation}
L(\mathcal{Q}_p')=0,\quad L(\Lambda \mathcal{Q}_p)=-2\mathcal{Q}_p.
\end{equation}
Finally, we denote by $\delta(p)$ a small positive constant such that:
\begin{equation}
\lim_{p\rightarrow5}\delta(p)=0.
\end{equation}

\subsection{Strategy of the proof}
We will give in this subsection a brief insight of the proof of Theorem \ref{MT}. We will first use the self-similar solution constructed by H. Koch in \cite{K}, to derive a finite dimensional dynamics, which fully describe the blow-up regime. Since we are considering the slightly supercritical case, it is helpful to view this equation as a perturbation of the critical equation in some sense. So we can use some critical techniques in our analysis, though they may have a totally different meaning in the supercritical case.
\subsubsection{Derivation of the law}
We look for a solution to \eqref{CP} of the form:
\begin{equation}
u(t,x)=\frac{1}{\lambda(t)^{\frac{2}{p-1}}}V_{b(t)}\bigg{(}\frac{x-x(t)}{\lambda(t)}\bigg{)},
\end{equation}
and introduce the rescaled time:
$$\frac{ds}{dt}=\frac{1}{\lambda(t)^3}.$$
Then $u$ is a solution to \eqref{CP} if and only if $V_b$ solves the following equation:
\begin{equation}
b_s\frac{\partial V_b}{\partial b}-\frac{\lambda_s}{\lambda}\Lambda V_b+(V_b''-V_b+V_b|V_b|^{p-1})'=\bigg{(}\frac{x_s}{\lambda}-1\bigg{)}V_b'.
\end{equation}
Similar to the Schr\"{o}dinger case, the self-similar blow-up regime of \eqref{CP} corresponds to the following finite dimensional dynamics:
\begin{equation}\label{17}
\frac{ds}{dt}=\frac{1}{\lambda^3},\quad\frac{x_s}{\lambda}=1,\quad \frac{\lambda_s}{\lambda}=-b,\quad b_s=0,
\end{equation}
which, after integrating, leads to finite time blow-up for $b(0)>0$ with:
$$\lambda(t)=c(u_0)\sqrt[3]{T-t}.$$
\subsubsection{Decomposition of the flow and modulation equations (section 2 and section 3)}From the previous discussing we can see it is significant to find a solution $Q_b$ to \eqref{PF}, which is done by H. Koch in \cite{K}. For our analysis, it is better to work with the localized approximate self-similar profile%
\footnote{See detailed discussing in Section 2.2.}%
:
$$Q_b(y)=v(b,p,y)\chi_0(b_cy).$$
Then we can introduce the geometrical decomposition of the flow:
$$u(t,x)=\frac{1}{\lambda(t)^{\frac{2}{p-1}}}\big{(}Q_{b(t)}+\varepsilon\big{)}\bigg{(}t,\frac{x-x(t)}{\lambda(t)}\bigg{)},$$
where the 3 time dependent parameters are adjusted to ensure suitable orthogonality conditions%
\footnote{See \eqref{OC}.}
 for $\varepsilon$. The modulation equations of the parameters are:
\begin{equation}\label{18}\begin{split}
&\frac{\lambda_s}{\lambda}+b=O(b_c^{\frac{5}{2}}+\|\varepsilon\|_{H^1_{\text{loc}}}),\\
&\frac{x_s}{\lambda}-1=O(b_c^{\frac{5}{2}}+\|\varepsilon\|_{H^1_{\text{loc}}}),\\
&b_s+c_p\tilde{b}b_c=O(b_c^3+b_c\|\varepsilon\|_{H^1_{\text{loc}}}).
\end{split}\end{equation}
Our main task here is to control $\|\varepsilon\|_{H^1_{\text{loc}}}$, which is done by a bootstrap argument%
\footnote{See Proposition \ref{BS}.}%
. If such a control exists, we will see that \eqref{18} is just a small perturbation of the system \eqref{17}, and has almost the same behavior%
\footnote{See detailed proof in Section 6.2.}%
.

\subsubsection{Monotonicity formula (section 4 and section 5)} The key techniques in this paper and the monotonicity of energy and a dispersive control of $\|\varepsilon\|_{H^1_{\text{loc}}}$.

The monotonicity of the energy gives a much better control of the $L^2$ norm of $\varepsilon_y$ on the half-line $[\kappa B,+\infty)$. Together with Gagliardo-Nirenberg inequality, we can control the localized $L^2$ norm of $\varepsilon$ on the right.

Next, we build a nonlinear functional:
$$\mathcal{F}\sim\int\bigg{[}\varepsilon_y^2\psi+\varepsilon^2\zeta-\frac{2}{p+1}\big{(}|\varepsilon+Q_b|^{p+1}-Q_b^{p+1}-(p+1)\varepsilon Q_b^p\big{)}\psi\bigg{]},$$
for well chosen functions $(\psi,\zeta)$, which are exponentially decaying to the left and bounded on the right. A similar functional was introduced in \cite{MMR1} for the critical equations, but they have a totally different meaning. Here the key point in supercritical case is that we cannot control $\int_{y>0}\varepsilon^2$. We must assume that $\zeta$ is compactly supported on the right, i.e. $\text{supp }\zeta\subset (-\infty,2B^2]$, for some large constant $B$. Then for $y>0$, only localized $L^2$ norm of $\varepsilon$ appears in $\mathcal{F}$, which can be controlled  by using the monotonicity of energy introduced before.

Moreover, from the choice of orthogonality conditions, the leading order term of $\mathcal{F}$ is coercive:
$$\mathcal{F}\sim\|\varepsilon\|_{H^1_{\text{loc}}}^2.$$

The most significant technique here is the $Lyapounov$ $monotonicity$:
\begin{equation}\label{19}
\frac{d\mathcal{F}}{ds}+\frac{1}{B}\|\varepsilon\|_{H^1_{\text{loc}}}^2\lesssim b_c^{\frac{7}{2}}.
\end{equation}
This formula shows that $\|\varepsilon\|_{H^1_{\text{loc}}}$ (or equivalently $\mathcal{F}$) is almost decreasing with respect to $s\in[0,+\infty)$. So it is controlled by a small constant (say, $b_c^{3+8\nu}$) if we choose a good initial data.

\subsubsection{End of the proof of Theorem \ref{MT}} We will see that the monotonicity formula \eqref{19} and modulation equations have already led to the bootstrap bound on $b$ and $\|\varepsilon\|_{H^1_{\text{loc}}}$. So we only need to prove the bound of $\|\varepsilon\|_{L^{p_0}}$. This is done by working on the original variable with the help of a refined Strichartz estimate%
\footnote{See Corollary \ref{RSEF}.}%
. Then we finish the bootstrap argument and the remaining part of Theorem \ref{MT} is followed by a standard procedure.

\subsection*{Acknowledgement} I would like to thank my supervisors F. Merle \& T. Duyckaerts for having suggested this problem to me and giving a lot of guidance.

\section{Description of the blow-up set of initial data}
This section is devoted to give a specific description of the open subset $\mathcal{O}_p$ of the initial data, which leads to the self-similar blow-up dynamics in Theorem 1.1. The most important part here is to construct a suitable approximate self-similar profile.
\subsection{Construction of the approximate self-similar profile}
This part follows H. Koch's work \cite{K}. To avoid misunderstanding, we use a different notation.

Let us consider a solution $u(t,x)$ of the form:
$$u(t,x)=\frac{1}{\big{(}3(T-t)\big{)}^{\frac{2}{3(p-1)}}}V\bigg{(}\frac{x}{\big{(}3(T-t)\big{)}^{\frac{1}{3}}}\bigg{)}.$$
Then by a standard computation, $u(t,x)$ is a solution if and only if $V(x)$ satisfies:
\begin{equation}\label{P}
\Lambda V+V'''+(V|V|^{p-1})'=0.
\end{equation}
For any constant $b>0$, we introduce a change of variable:
$$x=b^{\frac{1}{3}}(y+b^{-1}),\quad v(y)=b^{\frac{2}{3(p-1)}}V(b^{\frac{1}{3}}(y+b^{-1})).$$
Then \eqref{P} is equivalent to \eqref{PF}, i.e.
\begin{equation}\label{PF1}
b\Lambda v+(v''-v+v|v|^{p-1})'=0.
\end{equation}
The exact solution of \eqref{PF1} has been studied by H. Koch in \cite{K}. Actually H. Koch gives a even larger range of solutions.
\begin{proposition}[H. Koch \cite{K}]
There exist $p^*>5, b^*>0$, such that there exist 2 smooth maps: $\gamma(b,p): [0,b^*)\times[5,p^*)\rightarrow \mathbb{R}$, $v(b,p,y): [0,b^*)\times[5,p^*)\times\mathbb{R}\rightarrow \mathbb{R}$, such that the following holds:
\begin{enumerate}
\item The self-similar equation:
\begin{gather}
b\big{(}(1+\gamma(b,p))v+xv'\big{)}+(v''-v+v|v|^{p-1})'=0,\label{SFE}\\
 (v(b,p,\cdot),\mathcal{Q}_p'(\cdot))=0,\quad v(b,p,y)>0.
\end{gather}
\item For all $ p\in[5,p^*)$, there exists a unique $b=b(p)\in[0,b^*)$ such that:
\begin{equation}
\gamma(b(p),p)=-1+\frac{2}{p-1},\quad b(5)=0,
\end{equation}
Moreover,
\begin{gather}
\frac{d b(p)}{d p}\bigg{|}_{p=5}=\frac{\|\mathcal{Q}\|_{L^2}^2}{\|\mathcal{Q}\|_{L^1}^2}>0,\\
\frac{\partial\gamma}{\partial b}\bigg{|}_{b=b(p)}=-\frac{\|\mathcal{Q}_p\|_{L^1}^2}{8\|\mathcal{Q}_p\|_{L^2}^2}+O(|p-5|)<0,\\
\frac{1}{2}\int|v_y(b(p),p,y))|^2dy-\frac{1}{p+1}\int|v(b(p),p,y)|^{p+1}dy=0.
\end{gather}
\item $v(b,p,\cdot)\in \dot{H}^1\cap L^{p+1}$, $v(b,p,\cdot)\notin L^2$ if $b>0$ and $v(0,p,y)=\mathcal{Q}_p(y)$. Moreover, let
$$w_p(b,y)=v(b,p,y)-\mathcal{Q}_p(y),$$
then for all $k,n\in\mathbb{N}$ there holds:
\begin{gather}
|w_p(b,y)|\lesssim\begin{cases}
e^{-\frac{1}{3b}}(1+b^{-2/3}|1-by|)^{-1-\gamma} & \text{if $y > b^{-1}$},\\
b\exp(\frac{1}{3b}[(1-by)^{3/2}-1]) & \text{if $b^{-1}\geq y>0$},\\
b(1-by)^{-1-\gamma} & \text{if $y\leq0$},
\end{cases}\label{ASB}\\
|\partial_y^k\partial_b^nv|
\lesssim\begin{cases}
e^{-\frac{1}{3b}}(1+b^{-2/3}|1-by|)^{-1-\gamma-k} & \text{if $y > b^{-1}$},\\
\big{|}\partial_y^k\partial_b^n\big{(}\text{\normalfont Hi}_{\gamma}(b^{-2/3}(1-by))/\text{\normalfont Hi}_{\gamma}(b^{-2/3})\big{)}\big{|} & \text{if $b^{-1}\geq y>0$},\\
\big{|}\partial_y^k\partial_b^n\big{(}b(1-by)^{-1-\gamma}\big{)}\big{|}+e^y & \text{if $y\leq0$},
\end{cases}\label{ASB1}
\end{gather}
where
$$\text{\normalfont Hi}_{\gamma}(x)=\frac{1}{\pi}\int_0^{+\infty}\sigma^{\gamma}e^{-\frac{1}{3}\sigma^2+\sigma x}d\sigma.$$
\end{enumerate}
\end{proposition}
\begin{remark}
(1) and (2) in Proposition 2.1 correspond to Theorem 3 in \cite{K}. \eqref{ASB} corresponds to Proposition 12 in \cite{K}. \eqref{ASB1} corresponds to Proposition 15 in \cite{K}%
\footnote{Let's mention that there is a slight problem in the original statement of this estimate in \cite{K} (i.e. Proposition 15 in \cite{K}). And \eqref{ASB1} is the correct version.}%
.
\end{remark}
\begin{remark}
In \cite{K}, H. Koch gives the following asymptotic behavior of $\text{Hi}_{\gamma}$:
$$\text{Hi}_{\gamma}(x)=\bigg{(}\frac{1}{\sqrt{\pi}}|x|^{-\frac{1}{4}+\frac{\gamma}{2}}+O(|x|^{-\frac{7}{4}+\frac{\gamma}{2}})\bigg{)}e^{\frac{2}{3}x^{3/2}},\text{ as }x\rightarrow +\infty.$$
together with the fact that $\partial_x\text{Hi}_{\gamma}=\text{Hi}_{\gamma+1}$,  we have for $b^{-1}\geq y> 0$:
\begin{equation*}\begin{split}
&\quad\;\big{|}\partial_y^k\partial_b^n\big{(}\text{\normalfont Hi}_{\gamma}(b^{-2/3}(1-by))/\text{\normalfont Hi}_{\gamma}(b^{-2/3})\big{)}\big{|}\\
&\lesssim_{k,n} \exp\Big{(}\frac{1}{3b}[(1-by)^{3/2}-1]\Big{)}\leq e^{-\frac{y}{10}}.
\end{split}\end{equation*}
Hence \eqref{ASB1} reads:
\begin{equation}\label{ASB2}
|\partial_y^k\partial_b^nv|
\lesssim_{k,n}\begin{cases}
e^{-\frac{1}{3b}}(1+b^{-2/3}|1-by|)^{-1-\gamma-k} & \text{if $y > b^{-1}$},\\
e^{-y/10} & \text{if $b^{-1}\geq y>0$},\\
\big{|}\partial_y^k\partial_b^n\big{(}b(1-by)^{-1-\gamma}\big{)}\big{|}+e^y & \text{if $y\leq0$},
\end{cases}
\end{equation}
\end{remark}

Now we fix some $p\in(5,p^*)$, and denote
\begin{equation}b_c=b(p)\sim p-5, \quad\tilde{b}=b-b_c.\end{equation}
 From now on, we will focus on the case $|\tilde{b}|\ll b_c$.

The exact self-similar solution $v$ is not in $L^2$, which is not good for our analysis. We need to construct a suitable approximation of $v$. Fortunately,  we observe that though $v$ has a slowly decaying tail at infinity, it is with a small coefficient:
\begin{equation*}
v(y)\sim\begin{cases}
\frac{e^{-1/3b_c}}{|y|^{1+\gamma}}\quad &\text{as } y\rightarrow +\infty,\\
\frac{b_c^{-\gamma}}{|y|^{1+\gamma}}\quad &\text{as } y\rightarrow -\infty.
\end{cases}
\end{equation*}
So it is reasonable to consider a suitable cut-off of $v$. Choose a smooth cut-off function $\chi_0(y)$, such that $\chi_0(y)=0$ if $|y|>2$, $\chi_0(y)=1$ if $|y|<1$. Then we define the approximate self-similar profile $Q_b(y)$ as:
\begin{equation}\label{21}
Q_b(y)=v(b,p,y)\chi(y),
\end{equation}
where $\chi(y)=\chi_0(b_cy)$. We have the following properties of the approximate self-similar profile:
\begin{lemma}[Properties of the localized profile]\label{SSP} Assume that $b_c$ is small and $|\tilde{b}|\ll b_c$, then there holds:
\begin{enumerate}
\item Estimates on $Q_b$, for all $ k\in\mathbb{N}$, $q\in [1,+\infty]$:
\begin{align}
&|\partial^k_yQ_b(y)|\lesssim_k e^{-\frac{y}{10}}, \quad\text{for } y\geq 0,\\
&|\partial_y^k Q_b(y)|\lesssim_k e^{y}+b_c^{1+k}\mathbf{1}_{[-2b_c^{-1},0]}(y),\quad\text{for } y\leq 0,\\
&\|Q_b-\mathcal{Q}_p\|_{L^q}\lesssim b_c^{1-\frac{1}{q}},\quad \|(Q_b-\mathcal{Q}_p)_y\|_{L^2}\lesssim b_c.
\end{align}
Here $\mathbf{1}_I$ is the characteristic function of any interval $I$.
\item $Q_b$ is an approximate solution to \eqref{PF}: Let
\begin{equation}
-\Phi_b=b\Lambda Q_b+(Q_b''-Q_b+Q_b|Q_b|^{p-1})',
\end{equation}
then for $k=0,1$:
\begin{equation}\label{APP}
\partial_y^k\Phi_b=C_p\tilde{b}b_c\partial_y^kQ_b+O\big{(}|\tilde{b}|^2\partial_y^kQ_b+b_c^{2}\mathbf{1}_{[-2,-1]}(b_cy)+e^{-\frac{1}{10b_c}}\mathbf{1}_{[1,2]}(b_cy)\big{)},
\end{equation}
where $C_p=\frac{d\gamma}{db}\big{|}_{b=b_c}<0$.
\item Energy property of $Q_b$:
\begin{equation}
|E(Q_b)|\lesssim b_c^3+|\tilde{b}|.
\end{equation}
\item Properties of the first order term with respect to $b$:
let $P_b(y)=\frac{\partial Q_b}{\partial b}(y)$, then
\begin{equation}\label{25}
|P_b(y)|\lesssim e^{-\frac{y}{10}}\mathbf{1}_{\{y>0\}}(y)+\mathbf{1}_{[-2b_c^{-1},0]}(y).
\end{equation}
Furthermore, we have:
\begin{equation}\label{NV}
(P_b,\mathcal{Q}_p)=\frac{1}{16}\bigg{(}\int \mathcal{Q}_p\bigg{)}^2+O(|p-5|)>0.
\end{equation}
\end{enumerate}
\end{lemma}
\begin{proof}
(1) is a direct consequence of the asymptotic behavior of $v$, i.e. \eqref{ASB} and \eqref{ASB2}. For (2), a standard computation shows that:
\begin{multline*}\label{24}
-\Phi_b=-C_p\tilde{b}(b_c+\tilde{b})Q_b-b(\gamma+1-\frac{2}{p-1}-C_p\tilde{b})Q_b+\big{(}byv\chi'+v\chi'''\\+3v'\chi''
+3v''\chi'-v\chi'+pv'v^{p-1}(\chi^p-\chi)+p\chi'\chi^{p-1}v^p\big{)}.
\end{multline*}
Then (2) follows immediately from \eqref{ASB}, \eqref{ASB2} and the choice of $\chi$.

For (3), we note that $E(v(b_c,p,\cdot))=0$, and again from \eqref{ASB} we obtain:
$$|E(Q_b)-E(v(b_c,p,\cdot))|\lesssim |\tilde{b}|+b_c^3.$$

Finally we prove (4). First, \eqref{25} follows immediately from \eqref{ASB2}. For \eqref{NV}, we let $P(y)=\frac{\partial v}{\partial b}\big{|}_{b=0}(y)$.
From \eqref{ASB2} and continuity,
\begin{align*}
|P_b(y)-P(y)|&=\bigg{|}b\int_0^1\frac{\partial^2v}{\partial b^2}(tb,p,y)\chi(y)dt-P(y)\big{(}1-\chi(y)\big{)}\bigg{|}\\
&\lesssim b_c|y|\mathbf{1}_{[-2b_c^{-1},0]}(y)+b_c\mathbf{1}_{[-2b_c^{-1},2b_c^{-1}]}(y)+\mathbf{1}_{\{|y|>1/b_c\}}(y),
\end{align*}
which yields:
$$|(P_b,\mathcal{Q}_p)-(P,\mathcal{Q}_p)|\lesssim b_c=O(|p-5|).$$
So we only need to show that:
\begin{equation}\label{NV1}(P,\mathcal{Q}_p)=\frac{1}{16}\bigg{(}\int \mathcal{Q}_p\bigg{)}^2+O(|p-5|)>0.\end{equation}
We consider the Taylor's expansion of $v$ with respect to $b$ for $b\rightarrow0^+$ (here we ignore the assumption $|\tilde{b}|\ll b_c$). And then keep track of the first order term of $b$ in \eqref{SFE}. Observe that $\gamma(0,p)=\frac{2}{p-1}-1+O(|p-5|)$,  so we obtain:
\begin{align*}(LP)'=\Lambda \mathcal{Q}_p+O(|p-5|)\mathcal{Q}_p.
\end{align*}
Taking scalar product with $\int_{-\infty}^{y}\Lambda \mathcal{Q}_p$ yields
$$\frac{1}{2}\bigg{(}\int \Lambda\mathcal{Q}_p\bigg{)}^2+O(|p-5|)=-(LP,\Lambda \mathcal{Q}_p)=-(P,L(\Lambda \mathcal{Q}_p))=2(P,\mathcal{Q}_p).$$
Since $$\int \Lambda\mathcal{Q}_p=\bigg{(}\frac{2}{p-1}-1\bigg{)}\int \mathcal{Q}_p=\bigg{(}-\frac{1}{2}+O(|p-5|)\bigg{)}\int \mathcal{Q}_p,$$
then \eqref{NV1} follows, which concludes the proof of the Lemma.
\end{proof}

\subsection{Description of the blow-up set of initial data}
\begin{definition}\label{IDS}
Fix a small universal constant $\nu>0$ (which will be chosen later). For $p\in(5,p^*(\nu))$ with $p^*(\nu)$ close enough to $5$, we let $\mathcal{O}_p$ be the set of initial data $u_0\in H^1$ of the form:
$$u_0(x)=\frac{1}{\lambda_0^{\frac{2}{p-1}}}(Q_{b_0}+\varepsilon_0)\bigg{(}\frac{x-x_0}{\lambda_0}\bigg{)}$$
with parameter $(\lambda_0,x_0,b_0)\in\mathbb{R}_+^*\times\mathbb{R}\times\mathbb{R}_+^*$, such that:
\begin{enumerate}
\item $b_0$ is near $b_c(=b(p)\sim \sigma_c\sim p-5>0)$:
\begin{equation}
|b_0-b_c|< b_c^{\frac{7}{2}};
\end{equation}
\item Smallness of $\varepsilon_0$ in $H^1$:
\begin{equation}
\int \varepsilon_0^2+(\varepsilon_0)_y^2< b_c^{30};
\end{equation}
\item Condition on the scaling  parameter:
\begin{equation}
0<\lambda_0\leq1.
\end{equation}
\end{enumerate}
\end{definition}
\begin{remark}
It is easy to verify that $\mathcal{O}_p$ is nonempty. We may choose suitable $b_0,x_0,\lambda_0$, and set $\varepsilon_0=0$.
\end{remark}
\subsection{Setting the bootstrap}
Let $u_0\in \mathcal{O}_p$, and $u(t)$ be the corresponding solution to \eqref{CP} with maximal time interval $[0,T)$, $0<T\leq+\infty$. By using the regularity $u\in C([0,T),H^1)$ and a standard modulation theory%
\footnote{See Lemma 1 in \cite{MM4} and Lemma 2.5 in \cite{MMR1}.}%
(up to some small perturbations), we can find a  $0<T^*\leq T$, such that for all $ t\in[0,T^*)$, $u(t,x)$ admits a unique decomposition:
\begin{equation}\label{GD}
u(t,x)=\frac{1}{\lambda(t)^{\frac{2}{p-1}}}(Q_{b(t)}+\varepsilon(t))\bigg{(}\frac{x-x(t)}{\lambda(t)}\bigg{)}
\end{equation}
with geometrical parameters $(\lambda(t),x(t),b(t))\in\mathbb{R}_+^*\times\mathbb{R}\times\mathbb{R}_+^*$, which are all $C^1$ functions and the following orthogonality condition holds:
\begin{equation}\label{OC}
(\varepsilon(t),\mathcal{Q}_p)=(\varepsilon(t),\Lambda \mathcal{Q}_p)=(\varepsilon(t),y\Lambda \mathcal{Q}_p)=0.
\end{equation}
Moreover, we may assume that:
\begin{gather}
|\tilde{b}(0)|=|b(0)-b_c|\leq b_c^2,\\
\int \varepsilon^2(0)+\varepsilon_y^2(0)< b_c^{20}\label{HS},\\
0<\lambda(0)\leq 2.
\end{gather}

Now we state the bootstrap argument. Denote
\begin{equation}\label{B}
B=b_c^{-\frac{1}{20}}
\end{equation}
and then choose a smooth function  $\varphi$ such that:
\begin{equation}\begin{split}
&\varphi(y)=\begin{cases}
e^y & \text{for } y<-1,\\
1+y & \text{for } -\kappa<y<\kappa,\\
3 &\text{for } y>1,
\end{cases}\\
&\varphi'(y)\geq 0 \text{ for all } y\in\mathbb{R},
\end{split}\end{equation}
where $0<\kappa<1$ is a small universal constant to be chosen later%
\footnote{See in Appendix A.}%
.
We let
$\varphi_B(y)=\varphi(\frac{y}{B})$,
and define the localized Sobolev norm of $\varepsilon$:
\begin{equation}
\mathcal{N}(t)=B\bigg{(}\int \varepsilon^2(t,y)\varphi_B'(y)dy+\int \varepsilon_y^2(t,y)\varphi_B'(y)dy\bigg{)}.
\end{equation}
By continuity, we may assume that on $[0,T^*)$,  the following a {\it priori} bound holds:
\begin{gather}
|\tilde{b}(t)|\leq b_c^{\frac{3}{2}+\nu}\label{BS1},\\
\mathcal{N}(t)\leq b_c^{3+6\nu}\label{BS3},\\
\|\varepsilon(t)\|_{L^{p_0}}\leq b_c^{\frac{23}{50}}\label{BS4},\\
\|\varepsilon_y\|_{L^2}\leq b_c^{\frac{2}{3}}\label{BS5}.
\end{gather}
Here we choose
$$p_0=\frac{5}{2}.$$
\begin{remark}
From bootstrap assumption \eqref{BS4}, \eqref{BS5} and Gagliardo-Nirenberg inequality, we have for all $ q_0\geq p_0$,
\begin{equation*}
\|\varepsilon\|_{L^{q_0}}\lesssim \|\varepsilon\|_{L^{p_0}}^{\frac{p_0(q_0+2)}{q_0(p_0+2)}}\|\varepsilon_y\|_{L^2}^{\frac{2(q_0-p_0)}{q_0(p_0+2)}}\leq b_c^{\frac{149q_0-62}{270q_0}}.
\end{equation*}
In particular, for $q_0=p$ (note that $p$ is slightly larger than 5) and $q_0=+\infty$, we have:
\begin{equation}\label{CD2}
\int|\varepsilon|^p\lesssim b_c^{\frac{5}{2}},\quad \|\varepsilon\|_{L^{\infty}}\leq b_c^{\frac{149}{270}},
\end{equation}
Moreover, for all $ t\in[0,T^*)$:
\begin{equation}\label{CD1}
\int \varepsilon^2(t)e^{-\frac{|y|}{2}}\lesssim \mathcal{N}(t)+e^{-\kappa B/2}\|\varepsilon\|_{L^{\infty}}^2\leq b_c^{20}+\mathcal{N}(t).
\end{equation}
\end{remark}
Our main claim is that the above regime is trapped:
\begin{proposition}\label{BS}
There holds for all $ t\in[0,T^*)$,
\begin{gather}
|\tilde{b}(t)|\leq b_c^{\frac{3}{2}+2\nu}\label{BB1},\\
\mathcal{N}(t)\leq b_c^{3+8\nu}\label{BB3},\\
\|\varepsilon(t)\|_{L^{p_0}}\leq b_c^{\frac{13}{28}}\label{BB4},\\
\|\varepsilon_y\|_{L^2}\leq b_c^{\frac{3}{4}}\label{BB5}.
\end{gather}
and hence we may take $T^*=T$.
\end{proposition}

The next 3 sections are devoted to derive the dynamical controls of the geometrical parameters and monotonicity tools, which are the heart of the proof of the bootstrap bound in Proposition \ref{BS}. Then Theorem \ref{MT} is just a simple consequence of Proposition \ref{BS}, which will be shown in Section 6.

\section{Modulation equations}
In the framework of the geometrical decomposition \eqref{GD}, we introduce a new variable:
\begin{equation}
s=\int_0^t\frac{1}{\lambda^3(t')}dt',\quad y=\frac{x-x(t)}{\lambda(t)}.
\end{equation}
Now we use $(s,y)$ instead of the original variables $(t,x)$, and denote $s^*=s(T^*)$. Then we can claim the following properties:
\begin{proposition}The map $s\in[0,s^*)\rightarrow(\lambda(s),x(s),b(s))$ is $C^1$ and the following holds:
\begin{enumerate}
\item Equation of $\varepsilon$: for all $ s\in[0,s^*)$,
\begin{equation}\label{ME}\begin{split}
\varepsilon_s-(L\varepsilon)_y+b\Lambda\varepsilon=&\bigg{(}\frac{\lambda_s}{\lambda}+b\bigg{)}(\Lambda Q_b+\Lambda\varepsilon)+\bigg{(}\frac{x_s}{\lambda}-1\bigg{)}(Q_b+\varepsilon)_y\\
&+\Phi_b-b_sP_b-(R_b(\varepsilon))_y-(R_{NL}(\varepsilon))_y,
\end{split}\end{equation}
where
\begin{gather}
\Phi_b=-b\Lambda Q_b-(Q_b''-Q_b+Q_b^p)',\\
R_b(\varepsilon)=p(Q_b^{p-1}-\mathcal{Q}_p^{p-1})\varepsilon,\\
 R_{\text{NL}}(\varepsilon)=(\varepsilon+Q_b)|\varepsilon+Q_b|^{p-1}-p\varepsilon Q_b^{p-1}-Q_b^p.
\end{gather}
\item Modulation equation:
\begin{gather}
\bigg{|}\frac{\lambda_s}{\lambda}+b_c\bigg{|}\lesssim b_c^{\frac{5}{2}}+\mathcal{N}^{\frac{1}{2}}\label{MES1},\\
\bigg{|}\frac{x_s}{\lambda}-1\bigg{|}\lesssim b_c^{\frac{5}{2}}+\mathcal{N}^{\frac{1}{2}}\label{MES2},\\
|b_s+c_p\tilde{b}b_c|\lesssim b_c^{3}+b_c\mathcal{N}^{\frac{1}{2}}\label{MES3},
\end{gather}
where $c_p$ is a positive constant with $c_p=2+O(|p-5|)$.
\end{enumerate}
\end{proposition}

\begin{proof}
The proof of \eqref{ME} follows from a direct computation and the equation of $u(t)$.  Now we prove \eqref{MES1}--\eqref{MES3}. Let us differentiate the orthogonality condition $(\varepsilon,\Lambda \mathcal{Q}_p)=(\varepsilon,y\Lambda \mathcal{Q}_p)=0$ and use \eqref{CD2} to obtain:
\begin{equation*}\begin{split}
&\quad\;\bigg{|}\bigg{(}\frac{\lambda_s}{\lambda}+b\bigg{)}(\Lambda Q_b, \Lambda \mathcal{Q}_p)\bigg{|}+\bigg{|}\bigg{(}\frac{x_s}{\lambda}-1\bigg{)}(Q_b', y\Lambda \mathcal{Q}_p)\bigg{|}\\
&\lesssim \bigg{|}\bigg{(}\frac{\lambda_s}{\lambda}+b\bigg{)}(\Lambda Q_b,y \Lambda \mathcal{Q}_p)\bigg{|}+\bigg{|}\bigg{(}\frac{x_s}{\lambda}-1\bigg{)}(Q'_b, \Lambda \mathcal{Q}_p)\bigg{|}+b_c|\tilde{b}|+|b_s|\\
&\quad+\int\big{(}\varepsilon^2e^{-\frac{|y|}{2}}+|\varepsilon|^p\big{)}+b_c\bigg{(}\int\varepsilon^2e^{-\frac{|y|}{2}}\bigg{)}^{\frac{1}{2}}+\big{|}(\varepsilon, L(\Lambda \mathcal{Q}_p)')+(\varepsilon, L(y\Lambda \mathcal{Q}_p)')\big{|}\\
&\lesssim  \bigg{|}\bigg{(}\frac{\lambda_s}{\lambda}+b\bigg{)}(\Lambda Q_b,y \Lambda \mathcal{Q}_p)\bigg{|}+\bigg{|}\bigg{(}\frac{x_s}{\lambda}-1\bigg{)}(Q'_b, \Lambda \mathcal{Q}_p)\bigg{|}+b_c^{\frac{5}{2}}+|b_s|+\bigg{(}\int\varepsilon^2e^{-\frac{|y|}{2}}\bigg{)}^{\frac{1}{2}}.
\end{split}\end{equation*}
From \eqref{ASB}, we have for all $y\in\mathbb{R}$:
$$|Q_b(y)-\mathcal{Q}_p(y)|\lesssim b_c,$$
which implies:
\begin{align*}
\big{|}(\Lambda Q_b, \Lambda \mathcal{Q}_p)-(\Lambda \mathcal{Q}_p, \Lambda \mathcal{Q}_p)\big{|}\leq \|Q_b-\mathcal{Q}_p\|_{L^{\infty}}\|\Lambda^*\Lambda \mathcal{Q}_p\|_{L^1}=O(b_c).
\end{align*}
hence $(\Lambda Q_b, \Lambda \mathcal{Q}_p)=\|\Lambda \mathcal{Q}_p\|_{L^2}^2+O(b_c).$
Similarly, we have
\begin{equation*}
(Q_b', y\Lambda \mathcal{Q}_p)=\|\Lambda \mathcal{Q}_p\|_{L^2}^2+O(b_c),\quad
(\Lambda Q_b,y \Lambda \mathcal{Q}_p)=O(b_c),\quad(Q'_b, \Lambda \mathcal{Q}_p)=O(b_c).
\end{equation*}
Combining these estimates with \eqref{CD1} we have:
\begin{equation}\label{1}
\bigg{|}\frac{\lambda_s}{\lambda}+b\bigg{|}+\bigg{|}\frac{x_s}{\lambda}-1\bigg{|}\lesssim b_c^{\frac{5}{2}}+|b_s|+\mathcal{N}^{\frac{1}{2}}.
\end{equation}

Now we differentiate  the orthogonality condition $(\varepsilon,\mathcal{Q}_p)=0$. A similar computation shows:
\begin{equation}\label{2}\begin{split}
&\quad|(P_b,\mathcal{Q}_p)b_s-(\Phi_b,\mathcal{Q}_p)|\\
&\lesssim O(b_c)\bigg{(}\bigg{|}\frac{\lambda_s}{\lambda}+b\bigg{|}+\bigg{|}\frac{x_s}{\lambda}-1\bigg{|}\bigg{)}+\int(\varepsilon^2e^{-\frac{|y|}{2}}+|\varepsilon|^p)+b_c\Big{(}\int\varepsilon^2e^{-\frac{|y|}{2}}\Big{)}^{\frac{1}{2}}\\
&=O(b_c)\bigg{(}b_c^{\frac{5}{2}}+\mathcal{N}^{\frac{1}{2}}+\bigg{|}\frac{\lambda_s}{\lambda}+b\bigg{|}+\bigg{|}\frac{x_s}{\lambda}-1\bigg{|}\bigg{)}.
\end{split}\end{equation}
Observe from \eqref{APP}  and \eqref{NV}:
\begin{equation*}\begin{split}
&(\Lambda Q_b,\mathcal{Q}_p)=O(b_c),\quad (Q_b',\mathcal{Q}_p)=O(b_c),\\
&(P_b,\mathcal{Q}_p)=\frac{1}{16}\|\mathcal{Q}_p\|_{L^1}^2+O(|p-5|)>0,\\
&(\Phi_b,\mathcal{Q}_p)=C_pb_c\tilde{b}(Q_b,\mathcal{Q}_p)+O(|\tilde{b}|^2+e^{-\frac{1}{2b_c}})=-\tilde{c}_p\|\mathcal{Q}_p\|_{L^1}^2 b_c\tilde{b}+O(b_c^{3}),
\end{split}\end{equation*}
with $\tilde{c}_p=\frac{1}{8}+O(|p-5|)>0$. First, from \eqref{2} we have:
\begin{equation}\label{34}
|b_s|\lesssim b_c^{\frac{5}{2}}+O(b_c)\bigg{(}\mathcal{N}^{\frac{1}{2}}+\bigg{|}\frac{\lambda_s}{\lambda}+b\bigg{|}+\bigg{|}\frac{x_s}{\lambda}-1\bigg{|}\bigg{)}.
\end{equation}
Injecting \eqref{34} into \eqref{1}, we obtain \eqref{MES1} and \eqref{MES2}. Moreover \eqref{2} implies:
\begin{equation}\label{3}
|b_s+c_pb_c\tilde{b}|=O(b_c)\bigg{(}b_c^{2}+\mathcal{N}^{\frac{1}{2}}+\bigg{|}\frac{\lambda_s}{\lambda}+b\bigg{|}+\bigg{|}\frac{x_s}{\lambda}-1\bigg{|}\bigg{)},
\end{equation}
where $c_p=2+O(|p-5|)$. Then \eqref{MES3} follows from \eqref{MES1}, \eqref{MES2} and \eqref{3}, which concludes the proof of the proposition.
\end{proof}

\section{Monotonicity of the energy}
This section is devoted to derive a control of the $L^2$ norm of $\varepsilon_y$ by the energy conservation law and monotonicity. We will first give a control of $\|\varepsilon_y\|_{L^2}$ on the whole line, which proves the bootstrap bound \eqref{BB5}. But furthermore, we will show that on the half line $[\kappa B,+\infty)$, there is a much better bound for the $L^2$ norm of $\varepsilon_y$, which comes from the monotonicity of the localized energy%
\footnote{See \eqref{42}.}%
. Then by Gagliardo-Nirenberg inequality we can get a good control for the localized $L^2$ norm of $\varepsilon$.
\begin{lemma}
For all $ s\in[0,s^*)$, the following estimates hold:
\begin{gather}
\int\varepsilon^2_y(s)\lesssim b^{\frac{3}{2}+\nu}_c\label{ECL1},\\
\int_{y>\kappa B}\varepsilon_y^2(s)\lesssim b_c^{\frac{55}{7}}\label{ECL2}.
\end{gather}
\end{lemma}
\begin{remark}
\eqref{ECL1} is the desired bootstrap bound \eqref{BB5}.
\end{remark}
\begin{proof}[Proof of Lemma 4.1]
The first estimate \eqref{ECL1} is a consequence of the energy conservation law. We write down the energy equality explicitly:
\begin{multline}\label{ECL}
2\lambda(s)^{2(1-\sigma_c)}E(u_0)=2E(Q_b)+\int\varepsilon_y\big{(}Q_b-\mathcal{Q}_p\big{)}_y\\+\int\varepsilon_y^2-\int\varepsilon(\mathcal{Q}_p)_{yy}-\frac{2}{p+1}\int\big{(}(Q_b+\varepsilon)^{p+1}-Q_b^{p+1}\big{)}.\\
\end{multline}
From \eqref{BS3} and \eqref{MES1}, we know for all $ s\in[0,s^*)$
\begin{equation}
-(1+\nu)b_c\leq\frac{\lambda_s}{\lambda}\leq -(1-\nu)b_c<0.
\end{equation}
Therefore $\lambda(s)$ is decreasing on $[0,s^*)$, then we have:
\begin{equation*}\begin{split}
\int\varepsilon_y^2&\lesssim \lambda(s)^{2(1-\sigma_c)}|E(u_0)|+|\tilde{b}|+b_c^3+\|(Q_b-\mathcal{Q}_p)_y\|_{L^2}^2\\
&\quad+\bigg{(}\int \varepsilon^2e^{-\frac{|y|}{2}}\bigg{)}^{\frac{1}{2}}+\int (|\varepsilon|^{p}+Q_b^{p})|\varepsilon|\\
&\lesssim b_c^{\frac{3}{2}+\nu}+\lambda(0)^{2(1-\sigma_c)}|E(u_0)|+\int |\varepsilon|^{p+1}+\int_{y>\kappa B}Q_b^{p}|\varepsilon|\\
&\quad+\int_{|y|\leq\kappa B}Q_b^{p}|\varepsilon|+\int_{y<-\kappa B}Q_b^{p}|\varepsilon|\\
&\lesssim b_c^{\frac{3}{2}+\nu}+\lambda(0)^{2(1-\sigma_c)}|E(u_0)|+b_c^3+e^{-B}\bigg{(}\int_{y>\kappa B} \varepsilon^2e^{-\frac{|y|}{10}}\bigg{)}^{\frac{1}{2}}\\
&\quad+\bigg{(}\int_{|y|<\kappa B}|\varepsilon|^2\bigg{)}^{\frac{1}{2}}+\bigg{(}\int|\varepsilon|^{p_0}\bigg{)}^{\frac{1}{p_0}}\bigg{(}\int_{y<-\kappa B} Q_b^{pp_0'}\bigg{)}^{\frac{1}{p_0'}}\\
&\lesssim b_c^{\frac{3}{2}+\nu}+\lambda(0)^{2(1-\sigma_c)}|E(u_0)|.
\end{split}\end{equation*}
Here we use the fact that $|Q_b(y)|\lesssim b_c$, if $y<-\kappa B$, and $Q_b$ decays exponentially on the right.

So it remains to estimate $\lambda(0)^{2(1-\sigma_c)}|E(u_0)|$. We let $s=0$ in \eqref{ECL}, from the assumption of the initial data, we have:
\begin{equation*}
\lambda(0)^{2(1-\sigma_c)}|E(u_0)|\lesssim |E(Q_{b(0)})|+\|\varepsilon(0)\|_{H^1}\lesssim b_c^2+|\tilde{b}(0)|\lesssim b_c^{\frac{3}{2}+\nu},
\end{equation*}
then \eqref{ECL1} follows.

Now we prove \eqref{ECL2}. We use a bootstrap argument on $[0,T^*)$. We assume that for all $ t\in[0,T^*)$, we have:
\begin{equation}\label{400}
\int_{y>\kappa B}\varepsilon_y^2(t)\leq b_c^{\frac{15}{2}}.
\end{equation}
Since this estimate is satisfied for $t=0$, we only need to improve this estimate to:
\begin{equation}\label{401}
\int_{y>\kappa B}\varepsilon_y^2(t)\lesssim b_c^{\frac{55}{7}}\quad\text{for } \forall t\in[0,T^*).
\end{equation}
To do this we first choose a smooth function $\theta$ such that:
\begin{equation}
\theta(y)=e^{-|y|} \text{ for } |y|>1,\quad\theta(y)\geq\frac{1}{e}  \text{ for } |y|<1.
\end{equation}
We then define
$$\Theta(y)=\frac{1}{K}\int_{-\infty}^y\theta(y')dy',$$
where $K=\int_{-\infty}^{+\infty} \theta(y')dy'$.

Let $t\in[0,T^*)$ be any fixed time. For all $ \tau\in[0,t]$, we denote:
$$\tilde{x}(\tau)=\frac{1}{\sqrt{B}}\bigg{(}\frac{x-x(\tau)}{\lambda(\tau)}-\kappa B\bigg{)},\quad \tilde{y}=\frac{y-\kappa B}{\sqrt{B}},$$
$$\widetilde{E}(\tau)=\int\Big{(}\frac{1}{2}|u_x(\tau)|^2-\frac{1}{p+1}|u(\tau)|^{p+1}\Big{)}\Theta\big{(}\tilde{x}(\tau)\big{)}dx.$$
Observe that $\Theta(\tilde{y})\leq e^{-\frac{\kappa\sqrt{B}}{2}}\leq b_c^{20}$, if $y<\kappa B/2$, so we have:
\begin{equation}\label{41}\begin{split}
&\quad\lambda(t)^{2(1-\sigma_c)}\widetilde{E}(t)\\
&=\frac{1}{2}\int\big{(}(Q_b)_y+\varepsilon_y\big{)}^2\Theta(\tilde{y})dy-\frac{1}{p+1}\int|Q_b+\varepsilon|^{p+1}\Theta(\tilde{y})dy\\
&\gtrsim \int_{y>\kappa B}\varepsilon_y^2(t)-\int_{y>\frac{\kappa B}{2}}\Big{(}|(Q_b)_y|^2+|Q_b|^{p+1}\Big{)}-e^{-\frac{\kappa \sqrt{B}}{2}}\int_{y<\frac{\kappa B}{2}}\Big{(}|(Q_b)_y|^2+|Q_b|^{p+1}\Big{)}\\
&\quad-\int_{y>\frac{\kappa B}{2}}|\varepsilon|^{p+1}-e^{-\frac{\kappa \sqrt{B}}{2}}\int_{y<\frac{\kappa B}{2}}|\varepsilon|^{p+1}.
\end{split}\end{equation}

Next from \eqref{BS3}, \eqref{BS4}, \eqref{400} and localized Gagliardo-Nirenberg inequality, we know that (recall $p_0=\frac{5}{2}$):
\begin{align}
&\int_{y>\kappa B} |\varepsilon|^{p+1}\lesssim \bigg{(}\int|\varepsilon|^{p_0}\bigg{)}^{\frac{p+3}{p_0+2}}\bigg{(}\int_{y>\kappa B}\varepsilon_y^2\bigg{)}^{\frac{p+1-p_0}{p_0+2}}\lesssim b_c^{\frac{173p-156}{90}}\leq b_c^{\frac{55}{7}}\label{402},\\
&\|\varepsilon\|_{L^{\infty}(y>\kappa B)}\lesssim \bigg{(}\int|\varepsilon|^{p_0}\bigg{)}^{\frac{1}{p_0+2}}\bigg{(}\int_{y>\kappa B}\varepsilon_y^2\bigg{)}^{\frac{1}{p_0+2}}\leq b_c^{\frac{173}{90}}\leq b_c^{\frac{3}{2}}\label{406}.
\end{align}
On the other hand, by Sobolev embedding we can show:
\begin{equation}\label{405}
\|\varepsilon\|_{L^{\infty}(|y|<\kappa B)}\lesssim \mathcal{N}^{\frac{1}{2}}\leq b_c^{\frac{3}{2}},
\end{equation}
hence
\begin{equation}\label{403}
\int_{\frac{\kappa B}{2}<y<\kappa B}|\varepsilon|^{p+1}\leq \|\varepsilon\|^{p-1}_{L^{\infty}(|y|<\kappa B)}\bigg{(}\int_{|y|\leq \kappa B}\varepsilon^2\bigg{)}\leq b_c^9.
\end{equation}
Injecting \eqref{402} and \eqref{403} into \eqref{41} yields:
\begin{equation}\label{404}
\int_{y>\kappa B}\varepsilon_y^2\lesssim b_c^{\frac{55}{7}}+\lambda(t)^{2(1-\sigma_c)}\widetilde{E}(t).
\end{equation}
Therefore, it remains to estimate $\widetilde{E}(t)$. We first use Kato's Localization identity for energy to compute:
\begin{equation*}\begin{split}
\frac{d}{d{\tau}}\widetilde{E}(\tau)=&-\frac{1}{2}\int(u_{xx}+u|u|^{p-1})^2g_x-\int u_{xx}^2g_x\\
&+p\int u|u|^{p-2}u_x^2g_x+\frac{1}{2}\int u_x^2g_{xxx}\\
&-\frac{x_t(\tau)}{\sqrt{B}\lambda(\tau)}\int\Big{(}\frac{1}{2}|u_x(\tau)|^2-\frac{1}{p+1}|u(\tau)|^{p+1}\Big{)}\theta\big{(}\tilde{x}(\tau)\big{)}dx\\
&-\frac{\lambda_t(\tau)}{\sqrt{B}\lambda(\tau)}\int\Big{(}\frac{1}{2}|u_x(\tau)|^2-\frac{1}{p+1}|u(\tau)|^{p+1}\Big{)}\bigg{(}\frac{x-x(\tau)}{\lambda(\tau)}\bigg{)}\theta\big{(}\tilde{x}(\tau)\big{)}dx\\
=&I+II+III+IV,
\end{split}\end{equation*}
where $g(x,\tau)=\Theta\big{(}\tilde{x}(\tau)\big{)}$.

We claim that for some universal constant $C>0$, there holds:
\begin{equation}\label{42}
\frac{d}{d{\tau}}\widetilde{E}(\tau)\leq \frac{Cb_c^9}{\lambda(\tau)^{3+2(1-\sigma_c)}}
\end{equation}

First $I\leq0$, since $g$ is nondecreasing in $x$. We then deal with $III$ and $IV$.
From \eqref{MES1} and \eqref{MES2} we have:
\begin{equation*}
x_t\sim\frac{1}{\lambda^2},\quad \lambda_t\sim-\frac{b_c}{\lambda^2}.
\end{equation*}

For $III$, we use \eqref{406}, \eqref{405} and the fact that $|\theta(\tilde{y})|\leq e^{-\frac{\kappa \sqrt{B}}{2}}$, if $y\leq \kappa B/2$ to estimate:
\begin{equation}\label{43}\begin{split}
III  &\leq -\frac{1}{4\sqrt{B}\lambda^3(\tau)}\int|u_x(\tau)|^2\theta\big{(}\tilde{x}(\tau)\big{)}+\frac{C}{\sqrt{B}\lambda(\tau)^{3+2(1-\sigma_c)}}\int|\varepsilon(\tau)+Q_{b(\tau)}|^{p+1}\theta(\tilde{y})\\
&\leq -\frac{1}{4\sqrt{B}\lambda^3(\tau)}\int|u_x(\tau)|^2\theta\big{(}\tilde{x}(\tau)\big{)}\\
&\quad +\frac{1}{\sqrt{B}\lambda(\tau)^{3+2(1-\sigma_c)}}\bigg{(}\|\varepsilon\|_{L^{\infty}(y>\frac{\kappa B}{2})}^{p+1}\int_{y>\kappa B/2}\theta(\tilde{y})dy+e^{-\frac{\kappa \sqrt{B}}{2}}\int_{y<\kappa B/2}|\varepsilon|^{p+1}\bigg{)}\\
&\quad+\frac{1}{\sqrt{B}\lambda(\tau)^{3+2(1-\sigma_c)}}\bigg{(}e^{-\frac{\kappa \sqrt{B}}{2}}\int_{y<\kappa B/2}|Q_{b(\tau)}|^{p+1}+\int_{y>\kappa B/2}|Q_{b(\tau)}|^{p+1}\bigg{)}\\
&\leq -\frac{1}{4\sqrt{B}\lambda^3(\tau)}\int|u_x(\tau)|^2\theta\big{(}\tilde{x}(\tau)\big{)}+\frac{Cb_c^9}{\lambda(\tau)^{3+2(1-\sigma_c)}}.
\end{split}\end{equation}

For $IV$, similarly there holds:
\begin{equation*}\begin{split}
IV &\leq \frac{b_c}{\sqrt{B}\lambda^3(\tau)}\int|u_x(\tau)|^2\bigg{|}\frac{x-x(\tau)}{\lambda(\tau)}\bigg{|}\theta\big{(}\tilde{x}(\tau)\big{)}+\frac{Cb_c^{9}}{\lambda(\tau)^{3+2(1-\sigma_c)}}\\
&=\frac{b_c}{\sqrt{B}\lambda(\tau)^{3+2(1-\sigma_c)}}\int|y|\varepsilon_y^2(\tau)\theta(\tilde{y})+\frac{Cb_c^{9}}{\lambda(\tau)^{3+2(1-\sigma_c)}}.
\end{split}\end{equation*}
We then divide the integral $\int|y|\varepsilon_y^2(\tau)\theta(\tilde{y})$ into 2 parts: $\int_{|y-\kappa B|>B}$ and $\int_{|y-\kappa B|\leq B}$. \\For the first part, we have $|y\theta(\tilde{y})|\leq e^{-\frac{\kappa \sqrt{B}}{2}}$ on this region, hence:
\begin{equation*}
\int_{|y-\kappa B|>B}|y|\varepsilon_y^2(\tau)\theta(\tilde{y})\leq e^{-\frac{\kappa \sqrt{B}}{2}}\int\varepsilon_y^2(\tau)\leq Cb_c^{9}.
\end{equation*}
For another part, we have $|yb_c|\ll 1$ on this region, hence:
\begin{align*}
\frac{b_c}{\sqrt{B}\lambda(\tau)^{3+2(1-\sigma_c)}}\int_{|y-\kappa B|\leq B}|y|\varepsilon_y^2(\tau)\theta(\tilde{y})\leq \frac{1}{100\sqrt{B}\lambda^3(\tau)}\int\frac{1}{2}|u_x(\tau)|^2\theta\big{(}\tilde{x}(\tau)\big{)}.
\end{align*}
Collecting the above estimates, we obtain:
\begin{equation}\label{44}
IV \leq \frac{1}{100 \sqrt{B}\lambda^3(\tau)}\int|u_x(\tau)|^2\theta\big{(}\tilde{x}(\tau)\big{)}+\frac{Cb_c^{9}}{\lambda(\tau)^{3+2(1-\sigma_c)}}.
\end{equation}

Finally, we estimate $II$:
\begin{equation*}\begin{split}
II&\leq \frac{C}{\sqrt{B}\lambda(\tau)^{3+2(1-\sigma_c)}}\int|\varepsilon(\tau)+Q_{b(\tau)}|^{p-1}|\varepsilon_y(\tau)+(Q_{b(\tau)})_y|^2\theta(\tilde{y})\\
&\quad+\frac{C}{B^{\frac{3}{2}}\lambda^3(\tau)}\int|u_x(\tau)|^2\theta''\big{(}\tilde{x}(\tau)\big{)}\\
&=II_1+II_2.
\end{split}\end{equation*}
For the first term $II_1$, we divide the integral into 2 parts $\int_{y<\kappa B/2}$ and $\int_{y>\kappa B/2}$ as before, to obtain:
\begin{equation*}\begin{split}
II_1&\leq \frac{C}{\sqrt{B}\lambda(\tau)^{3+2(1-\sigma_c)}}\int\theta(\tilde{y})\Big{(}|\varepsilon|^{p-1}\varepsilon_y^2+|\varepsilon|^{p-1}|Q_b'|^2+|Q_b|^{p-1}\big{(}|\varepsilon_y|^2+|Q_b'|^2\big{)}\Big{)}\\
& \leq \frac{C}{\sqrt{B}\lambda(\tau)^{3+2(1-\sigma_c)}}\bigg{(}\|\varepsilon\|^{p-1}_{L^{\infty}(y>\frac{\kappa B}{2})}\int_{y>\frac{\kappa B}{2}}\varepsilon^2_y+e^{-\frac{\kappa \sqrt{B}}{2}}\int_{y<\frac{\kappa B}{2}}\big{(}|(Q_b)_y|^2+\varepsilon_y^2\big{)}\\
&\qquad\qquad\qquad\qquad\quad+\int_{y>\frac{\kappa B}{2}}|(Q_b)_y|^2|Q_b|^{p-1}+\int_{y>\frac{\kappa B}{2}}e^{-\frac{|y|}{10}}\big{(}|\varepsilon|^{p-1}+\varepsilon_y^2\big{)}dy\bigg{)}\\
&\leq \frac{C}{\sqrt{B}\lambda(\tau)^{3+2(1-\sigma_c)}}\bigg{(}\|\varepsilon\|^{p-1}_{L^{\infty}(y>\frac{\kappa B}{2})}\int_{y>\frac{\kappa B}{2}}\varepsilon^2_y(\tau)+b_c^9\bigg{)}.
\end{split}\end{equation*}
Then from \eqref{406}, \eqref{405} and the fact that:
\begin{align*}
\int_{y>\kappa B/2}\varepsilon^2_y(\tau)\leq \int_{\kappa B>y>\kappa B/2}\varepsilon^2_y(\tau)+\int_{y>\kappa B}\varepsilon^2_y(\tau)\leq b_c^3,
\end{align*}
we obtain:
$$\|\varepsilon\|^{p-1}_{L^{\infty}(y>\kappa B/2)}\int_{y>\kappa B/2}\varepsilon^2_y(\tau)\lesssim b_c^{\frac{3(p-1)}{2}}\times b_c^3\leq b_c^9,$$
hence
\begin{equation}\label{45}
II_1\leq\frac{Cb_c^{9}}{\lambda(\tau)^{3+2(1-\sigma_c)}}.
\end{equation}
For the second term $II_2$, from the definition of $\theta$, we have $|\theta''|\lesssim \theta$, hence:
\begin{equation}\label{46}
II_2\leq \frac{1}{100\sqrt{B}\lambda^3(\tau)}\int|u_x(\tau)|^2\theta\big{(}\tilde{x}(\tau)\big{)}.
\end{equation}
Collecting \eqref{43}, \eqref{44}, \eqref{45} and \eqref{46}, we obtain \eqref{42}.

Observe that for $\beta>3$ there holds:
\begin{equation}\label{48}
\int_0^t\frac{1}{\lambda^{\beta}(\tau)}d\tau\leq -2\int_0^t\frac{\lambda_t(\tau)}{b_c\lambda^{\beta-2}(\tau)}d\tau\leq \frac{2}{(\beta-3)b_c\lambda^{\beta-3}(t)}.
\end{equation}
Integrating \eqref{42} from $0$ to $t$ yields:
\begin{equation}\label{47}\begin{split}
\lambda(t)^{2(1-\sigma_c)}\widetilde{E}(t)&\lesssim\lambda(t)^{2(1-\sigma_c)}\widetilde{E}(0)+b_c^{8}\lesssim\lambda(0)^{2(1-\sigma_c)}\widetilde{E}(0)+b_c^{8}\\
&\lesssim \int|\varepsilon_y(0)+(Q_{b(0)})_y|^2\theta(\tilde{y})+b_c^{8}\\
&\lesssim \int|(Q_{b(0)})_y|^2\theta(\tilde{y})dy+\|\varepsilon_y(0)\|_{L^2}^2+b_c^8\\
&\lesssim b_c^8,
\end{split}\end{equation}
where we use the assumption on the initial data, i.e. \eqref{HS}. Then \eqref{401} follows from \eqref{404} and \eqref{47}, which completes the proof of Lemma 4.1.
\end{proof}
\begin{remark}
From \eqref{ECL2} and Localized Gagliardo-Nirenberg inequality, we have the following $L^{\infty}$ estimate of $\varepsilon$:
\begin{equation}\label{49}
\|\varepsilon\|_{L^{\infty}(y>\kappa B)}\lesssim \bigg{(}\int|\varepsilon|^{p_0}\bigg{)}^{\frac{1}{p_0+2}}\bigg{(}\int_{y>\kappa B}\varepsilon_y^2\bigg{)}^{\frac{1}{p_0+2}}\leq b_c^{\frac{1261}{630}}\leq b_c^2,
\end{equation}
which is important in the derivation of the second monotonicity formula in the next section.
\end{remark}

\section{The second monotonicity formula}
This section is devoted to derive a second monotonicity tool for $\varepsilon$, which is the key technique to our analysis. It is a Lyapunov functional based on a suitable localised Hamiltonian which is somehow similar to that of \cite{MMR1}.  But here, due to the super-criticality, we cannot estimate the $L^2$ norm of $\varepsilon$ even on the half-line $(1/b_c,+\infty)$. We need to cut it off while this will generate some new terms to be controlled. But these new terms will be controlled by using the monotonicity of the energy introduced in the previous section.
\subsection*{Pointwise monotonicity}Recall from (2.32), the definition of $\varphi$. We let $\psi$, $\eta$ be another 2 smooth functions such that:
\begin{align}
&\psi(y)=\begin{cases}
e^{y} & \text{ for } y<-1,\\
1 & \text{ for } y>-\kappa,
\end{cases}\quad \psi'\geq0,\\
&\eta(y)=\begin{cases}
1 & \text{ for } y<1,\\
0 & \text{ for } y>2,
\end{cases}\quad \eta'\leq0.
\end{align}
Here, we observe that $\psi(-\kappa)=\varphi(-\kappa)+\kappa$, and $\psi(y)=\varphi(y)$ for all $y<-1$, so we may assume in addition:
\begin{equation}\label{APOW}
\varphi(y)\leq \psi(y)\leq (1+3\kappa)\varphi(y), \text{ for all }y\leq-\kappa.
\end{equation}
\begin{remark}
It is easy to check that for every $\frac{1}{2}>\kappa>0$, such $\psi$ and $\varphi$ exist.
\end{remark}
Now, recall $B=b_c^{-\frac{1}{20}}$. We let
$$\psi_B(y)=\psi(\frac{y}{B}),\quad\eta_B(y)=\eta(\frac{y}{B^2}),\quad \zeta_B(y)=\varphi_B\eta_B.$$
 and then define the following Lyapunov functional for $\varepsilon$:
\begin{equation}
\mathcal{F}=\int\bigg{[}\varepsilon_y^2\psi_B+\varepsilon^2\zeta_B-\frac{2}{p+1}\big{(}|\varepsilon+Q_b|^{p+1}-Q_b^{p+1}-(p+1)\varepsilon Q_b^p\big{)}\psi_B\bigg{]}.
\end{equation}
Our main goal here is the following monotonicity formula of $\mathcal{F}$:
\begin{proposition}[The second monotonicity formula]\label{MF}
There exists a universal constant $\mu>0$ such that for all $ s\in[0,s^*)$, the following holds:
\begin{enumerate}
\item Lyapunov control:
\begin{equation}\label{MF1}
\frac{d}{ds}\mathcal{F}+\mu\int\big{(}\varepsilon_y^2+\varepsilon^2\big{)}\varphi'_B\lesssim b_c^{\frac{7}{2}};
\end{equation}
\item Coercivity of $\mathcal{F}$:
\begin{equation}\label{MF2}
\mathcal{N}-b_c^{\frac{7}{2}}\lesssim\mathcal{F}\lesssim \mathcal{N}+b_c^{\frac{7}{2}}.
\end{equation}
\end{enumerate}
\end{proposition}
\begin{remark}
The proof of Proposition \ref{MF} is almost parallel to that of Proposition 3.1 in \cite{MMR1}. But since we have a control of the global $L^2$ norm of $\varepsilon$ (consequently the $L^{\infty}$ norm of $\varepsilon$)
, some part of the proof will be easier. 
\end{remark}
\begin{proof}[Proof of Proposition \ref{MF}]We will prove \eqref{MF1} and \eqref{MF2} in several steps:
\subsection*{Step 1 \normalfont{Algebraic computation of $\mathcal{F}$}} A direct computation shows:
\begin{equation*}
\begin{split}
\frac{d}{ds}\mathcal{F}=&2\int\psi_B(\varepsilon_y)_s\varepsilon_y+\varepsilon_s\Big{\{}\varepsilon\zeta_B-\psi_B\big{[}(\varepsilon+Q_b)|\varepsilon+Q_b|^{p-1}-Q_b^{p}\big{]}\Big{\}}\\
&-2\int\psi_B(Q_b)_s\big{[}(\varepsilon+Q_b)|\varepsilon+Q_b|^{p-1}-Q_b^{p}-p\varepsilon Q_b^{p-1}\big{]}\\
=&f_1+f_2+f_3,
\end{split}
\end{equation*}
where
\begin{align*}
&f_1=2\int\bigg{(}\varepsilon_s-\frac{\lambda_s}{\lambda}\Lambda\varepsilon\bigg{)}\Big{\{}-(\psi_B\varepsilon_y)_y+\varepsilon\zeta_B-\psi_B\big{[}(\varepsilon+Q_b)|\varepsilon+Q_b|^{p-1}-Q_b^{p}\big{]}\Big{\}},\\
&f_2=2\frac{\lambda_s}{\lambda}\int\Lambda\varepsilon\Big{\{}-(\psi_B\varepsilon_y)_y+\varepsilon\zeta_B-\psi_B\big{[}(\varepsilon+Q_b)|\varepsilon+Q_b|^{p-1}-Q_b^{p}\big{]}\Big{\}},\\
&f_3=-2\int\psi_B(Q_b)_s\big{[}(\varepsilon+Q_b)|\varepsilon+Q_b|^{p-1}-Q_b^{p}-p\varepsilon Q_b^{p-1}\big{]}.
\end{align*}
We claim that the following estimates hold for some universal constant $\mu_0>0$:
\begin{align}
&f_1\leq -\mu_0\int(\varepsilon_y^2+\varepsilon^2)\varphi_B'+Cb_c^{\frac{7}{2}}\label{51},\\
&f_k\leq \frac{\mu_0}{10}\int(\varepsilon_y^2+\varepsilon^2)\varphi_B'+Cb_c^{\frac{7}{2}}\label{52},\quad \text{for }k=2,3.
\end{align}
It is obvious that \eqref{MF1} follows from \eqref{51} and \eqref{52}.

In step 2 - step 5, we will prove \eqref{51} and \eqref{52}. Observe that the definition of $\varphi$, $\psi$ and $\zeta_B$ imply:
\begin{align}
&\text{for } \forall y\in(-\infty,\kappa],\quad|\varphi'''|+|\varphi''|+|\varphi|+|\psi'''|+|\psi'|+|\psi|\lesssim\varphi'\lesssim\varphi,\\
&\zeta_B'=
\begin{cases}
3\eta_B'&\text{ for }y>B^2,\\
0&\text{ for }B<y\leq B^2,\\
\varphi_B'&\text{ for }y<B.
\end{cases}
\end{align}
We will use these properties several times during the proof.
\subsection*{Step 2 \normalfont{Control of $f_1$}} We give the proof of \eqref{51} by using the equation \eqref{ME} in the following form:
\begin{multline}\label{EQVE}
\varepsilon_s-\frac{\lambda_s}{\lambda}\Lambda \varepsilon=\big{(}-\varepsilon_{yy}+\varepsilon-(\varepsilon+Q_b)|\varepsilon+Q_b|^{p-1}+Q_b^{p}\big{)}_y\\
+\bigg{(}\frac{\lambda_s}{\lambda}+b\bigg{)}\Lambda Q_b+\bigg{(}\frac{x_s}{\lambda}-1\bigg{)}(Q_b+\varepsilon)_y-b_sP_b+\Phi_b,
\end{multline}
where $$\Phi_b=-b\Lambda Q_b-(Q_b''-Q_b+Q_b^{p})', \quad P_b=\frac{\partial Q_b}{\partial b}.$$
Injecting \eqref{EQVE} into the definition of $f_1$ yields:
$$f_1=f_{1,1}+f_{1,2}+f_{1,3}+f_{1,4}+f_{1,5}$$
with
\begin{equation*}\begin{split}
f_{1,1}&=2\int\big{(}-\varepsilon_{yy}+\varepsilon-(\varepsilon+Q_b)|\varepsilon+Q_b|^{p-1}+Q_b^{p}\big{)}_y\Big{\{}-(\psi_B\varepsilon_y)_y\\
&\qquad\quad+\varepsilon\zeta_B-\psi_B\big{[}(\varepsilon+Q_b)|\varepsilon+Q_b|^{p-1}-Q_b^{p}\big{]}\Big{\}},\\
f_{1,2}&=2\bigg{(}\frac{\lambda_s}{\lambda}+b\bigg{)}\int\Lambda Q_b\Big{\{}-\psi_B\big{[}(\varepsilon+Q_b)|\varepsilon+Q_b|^{p-1}-Q_b^{p}\big{]}-(\psi_B\varepsilon_y)_y+\varepsilon\zeta_B\Big{\}},\\
f_{1,3}&=2\bigg{(}\frac{x_s}{\lambda}-1\bigg{)}\int(Q_b+\varepsilon)_y\Big{\{}-\psi_B\big{[}(\varepsilon+Q_b)|\varepsilon+Q_b|^{p-1}-Q_b^{p}\big{]}\\
&\qquad\qquad\qquad\quad-(\psi_B\varepsilon_y)_y+\varepsilon\zeta_B\Big{\}},\\
f_{1,4}&=-2b_s\int P_b\Big{\{}-(\psi_B\varepsilon_y)_y+\varepsilon\zeta_B-\psi_B\big{[}(\varepsilon+Q_b)|\varepsilon+Q_b|^{p-1}-Q_b^{p}\big{]}\Big{\}},\\
f_{1,5}&=2\int\Phi_b\Big{\{}-(\psi_B\varepsilon_y)_y+\varepsilon\zeta_B-\psi_B\big{[}(\varepsilon+Q_b)|\varepsilon+Q_b|^{p-1}-Q_b^{p}\big{]}\Big{\}}.
\end{split}\end{equation*}
\subsubsection*{\underline{Term $f_{1,1}$}\normalfont{:}} Let us integrate by parts to obtain a more manageable formula%
\footnote{See a similar computation in the proof of Proposition 3.1 in \cite{MMR1}.}%
:
\begin{equation*}\begin{split}
 f_{1,1}=&2\int \big{[}-\varepsilon_{yy}+\varepsilon-(\varepsilon+Q_b)|\varepsilon+Q_b|^{p-1}+Q_b^{p}\big{]}_y\big{(}-\psi'_B\varepsilon_y+\varepsilon(\zeta_B-\psi_B)\big{)}\\
 &+2\int\big{[}-\varepsilon_{yy}+\varepsilon-(\varepsilon+Q_b)|\varepsilon+Q_b|^{p-1}+Q_b^{p}\big{]}_y\\
 &\qquad \times\big{[}-\varepsilon_{yy}+\varepsilon-(\varepsilon+Q_b)|\varepsilon+Q_b|^{p-1}+Q_b^{p}\big{]}\psi_B.
\end{split}\end{equation*}
We compute these terms separately. First we integrate by parts to obtain:
\begin{multline*}
2\int[-\varepsilon_{yy}+\varepsilon]_y\big{[}-\psi'_B\varepsilon_y+\varepsilon(\zeta_B-\psi_B)\big{]}
=-2\bigg{\{}\int\psi'_B\varepsilon_{yy}^2\\+\int\varepsilon_y^2\big{(}\frac{3}{2}\zeta'_B-\frac{1}{2}\psi'_B-\frac{1}{2}\psi'''_B\big{)}+\int\varepsilon^2\big{(}\frac{1}{2}(\zeta'_B-\psi'_B)-\frac{1}{2}(\zeta_B-\psi_B)'''\big{)}\bigg{\}}
\end{multline*}
and
\begin{equation*}\begin{split}
&-2\int\big{[}(Q_b+\varepsilon)|Q_b+\varepsilon|^{p-1}-Q_b^{p}\big{]}_y(\zeta_B-\psi_B)\varepsilon\\
=&-2\int(\zeta_B-\psi_B)(Q_b)_y\big{[}(\varepsilon+Q_b)|\varepsilon+Q_b|^{p-1}-Q_b^{p}-p\varepsilon Q_b^{p-1}\big{]}\\
&-\frac{2}{p+1}\int(\zeta_B-\psi_B)'\big{[}|Q_b+\varepsilon|^{p+1}-Q_b^{p+1}-(p+1)\varepsilon Q_b^{p}\big{]}\\
&+2\int(\zeta_B-\psi_B)'\big{[}(\varepsilon+Q_b)|\varepsilon+Q_b|^{p-1}-Q_b^{p}\big{]}\varepsilon.
\end{split}\end{equation*}
Next by direct expansion:
\begin{multline*}\int\big{[}(\varepsilon+Q_b)|\varepsilon+Q_b|^{p-1}-Q_b^{p}\big{]}_y\psi'_B\varepsilon_y=\\
p\int\psi'_B\varepsilon_y\Big{\{}(Q_b)_y\big{[}|Q_b+\varepsilon|^{p-1}-Q_b^{p-1}\big{]}+|Q_b+\varepsilon|^{p-1}\varepsilon_y\Big{\}}.
\end{multline*}
Finally,
\begin{equation*}\begin{split}
&2\int\big{[}-\varepsilon_{yy}+\varepsilon-(\varepsilon+Q_b)|\varepsilon+Q_b|^{p-1}+Q_b^{p}\big{]}_y\\
&\quad \times\big{[}-\varepsilon_{yy}+\varepsilon-(\varepsilon+Q_b)|\varepsilon+Q_b|^{p-1}+Q_b^{p}\big{]}\psi_B\\
=&-\int \psi'_B\big{[}-\varepsilon_{yy}+\varepsilon-(\varepsilon+Q_b)|\varepsilon+Q_b|^{p-1}+Q_b^{p}\big{]}^2\\
=&-\int \psi'_B\Big{\{}\big{[}-\varepsilon_{yy}+\varepsilon-(\varepsilon+Q_b)|\varepsilon+Q_b|^{p-1}+Q_b^{p}\big{]}^2-[-\varepsilon_{yy}+\varepsilon]^2\Big{\}}\\
&-\int\psi'_B[-\varepsilon_{yy}+\varepsilon]^2\\
=&-\int \psi'_B\Big{\{}\big{[}-\varepsilon_{yy}+\varepsilon-(\varepsilon+Q_b)|\varepsilon+Q_b|^{p-1}+Q_b^{p}\big{]}^2-[-\varepsilon_{yy}+\varepsilon]^2\Big{\}}\\
&-\bigg{[}\int\psi'_B(\varepsilon_{yy}^2+2\varepsilon_y^2)+\int\varepsilon^2(\psi'_B-\psi'''_B)\bigg{]}.
\end{split}\end{equation*}

We collect all the above computations and obtain the following:
\begin{equation*}\begin{split}
\quad f_{1,1}=&-\int\big{[}3\psi'_B\varepsilon_{yy}^2+(3\zeta'_B+\psi'_B-\psi'''_B)\varepsilon_y^2+(\zeta'_B-\zeta'''_B)\varepsilon^2\big{]}\\
&-2\int\bigg{[}\frac{|\varepsilon+Q_b|^{p+1}-Q_b^{p+1}}{p+1}-\varepsilon Q_b^p-\varepsilon\big{(}(\varepsilon+Q_b)|\varepsilon+Q_b|^{p-1}-Q_b^p\big{)}\bigg{]}(\zeta'_B-\psi'_B)\\
&+2\int\big{[}(\varepsilon+Q_b)|\varepsilon+Q_b|^{p-1}-Q_b^{p}-p\varepsilon Q_b^{p-1}\big{]}(Q_b)_y(\psi_B-\zeta_B)\\
&+2p\int\psi'_B\varepsilon_y\big{\{}(Q_b)_y[|Q_b+\varepsilon|^{p-1}-Q_b^{p-1}]+|Q_b+\varepsilon|^{p-1}\varepsilon_y\big{\}}\\
&-\int\psi'_B\Big{\{}\big{[}-\varepsilon_{yy}+\varepsilon-\big{(}(\varepsilon+Q_b)|\varepsilon+Q_b|^{p-1}-Q_b^{p}\big{)}\big{]}^2-[-\varepsilon_{yy}+\varepsilon]^2\Big{\}}\\
=&(f_{1,1})^{<}+(f_{1,1})^{\sim}+(f_{1,1})^{>},
\end{split}\end{equation*}
where $(f_{1,1})^{<,\sim,>}$ correspond to the integration on $y<-\kappa B$, $|y|<\kappa B$ and $y>\kappa B$, respectively.

In the region $y>\kappa B$, we have $\psi'_B=\psi'''_B\equiv0$. From \eqref{ECL2}, \eqref{402} and \eqref{49}, we have:
\begin{equation*}\begin{split}
&\quad\;\Bigg{|}\int_{y>\kappa B}\big{[}3\psi'_B\varepsilon_{yy}^2+(3\zeta'_B+\psi'_B-\psi'''_B)\varepsilon_y^2+(\zeta'_B-\zeta'''_B)\varepsilon^2\big{]}\Bigg{|}\\
&\lesssim \int_{y>\kappa B}\varepsilon_y^2+\frac{1}{B}\int_{\kappa B<y<2B^2}\varepsilon^2\lesssim b_c^4+B\|\varepsilon\|_{L^{\infty}(y>\kappa B)}^2\\
&\lesssim Bb_c^4+b_c^4\leq b_c^{\frac{7}{2}}.
\end{split}\end{equation*}
Together with
\begin{equation*}\begin{split}
&\quad\;\Bigg{|}\int_{y>\kappa B}\bigg{[}\frac{|\varepsilon+Q_b|^{p+1}-Q_b^{p+1}}{p+1}-\varepsilon Q_b^p-\varepsilon\big{(}(\varepsilon+Q_b)|\varepsilon+Q_b|^{p-1}-Q_b^{p}\big{)}\bigg{]}(\zeta'_B-\psi'_B)\Bigg{|}\\
&\lesssim\int_{y>\kappa B}|\varepsilon|^{p+1}+|Q_b|^{p-1}\varepsilon^2\lesssim \|\varepsilon\|_{L^{p_0}}^{\frac{p_0(p+3)}{p_0+2}}\|\varepsilon_y\|_{L^2(y>\kappa B)}^{\frac{2(p+1-p_0)}{p_0+2}}+e^{-\frac{\kappa B}{20}}\|\varepsilon\|_{L^{\infty}}^2\\&\lesssim b_c^{\frac{7}{2}}
\end{split}\end{equation*}
and
\begin{equation*}\begin{split}
&\quad\Bigg{|}\int_{y>\kappa B}\big{[}(\varepsilon+Q_b)|\varepsilon+Q_b|^{p-1}-Q_b^{p}-p\varepsilon Q_b^{p-1}\big{]}(Q_b)_y(\psi_B-\zeta_B)\Bigg{|}\\
&\lesssim e^{-\frac{\kappa B}{20}}\bigg{(}\int_{y>\kappa B}\big{(}\varepsilon^2 e^{-\frac{|y|}{2}}+|\varepsilon|^p\big{)}\bigg{)}\lesssim b_c^{\frac{7}{2}},
\end{split}\end{equation*}
we obtain:
\begin{equation}\label{53}
(f_{1,1})^{>}\lesssim b_c^{\frac{7}{2}}.
\end{equation}
In the region $|y|<\kappa B$, $\zeta_B(y)=\varphi_B(y)=1+y/B$ and $\psi_B(y)=1$. In particular, $\zeta'''_B=\psi'_B=0$. We obtain:
\begin{equation*}\begin{split}
(f_{1,1})^{\sim}=&-\frac{1}{B}\int_{|y|<\kappa B}\bigg{\{}3\varepsilon_y^3+\varepsilon^2+2\bigg{[}\frac{|\varepsilon+Q_b|^{p+1}-Q_b^{p+1}}{p+1}\\
&\quad\quad\qquad\quad\quad-\varepsilon Q_b^p-\varepsilon\big{(}(\varepsilon+Q_b)|\varepsilon+Q_b|^{p-1}-Q_b^{p}\big{)}\bigg{]}\\
&\qquad\qquad\qquad+2\big{[}(\varepsilon+Q_b)|\varepsilon+Q_b|^{p-1}-Q_b^{p}-p\varepsilon Q_b^{p-1}\big{]}y(Q_b)_y\bigg{\}}\\
=&-\frac{1}{B}\int_{|y|<\kappa B}\Big{\{}3\varepsilon^2_y+\varepsilon^2-p\mathcal{Q}_p^{p-1}\varepsilon^2+p(p-1)y\mathcal{Q}_p'\mathcal{Q}_p^{p-2}\varepsilon^2\Big{\}}+R(\varepsilon),
\end{split}\end{equation*}
where
\begin{equation*}\begin{split}
R(\varepsilon)=-\frac{1}{B}&\int_{|y|<\kappa B}\bigg{\{}-p(Q_b^{p-1}-\mathcal{Q}_p^{p-1})\varepsilon^2+p(p-1)y\big{(}(Q_b)_yQ_b^{p-2}-\mathcal{Q}_p'\mathcal{Q}_p^{p-2}\big{)}\varepsilon^2\\
&+2\bigg{(}\frac{|\varepsilon+Q_b|^{p+1}-Q_b^{p+1}}{p+1}-\varepsilon Q_b^p-\frac{p}{2}Q_b^{p-1}\varepsilon^2\bigg{)}\\
&-2\varepsilon\Big{(}(\varepsilon+Q_b)|\varepsilon+Q_b|^{p-1}-Q_b^{p}-p\varepsilon Q_b^{p-1}\Big{)}\\
&+2\Big{[}(\varepsilon+Q_b)|\varepsilon+Q_b|^{p-1}-Q_b^{p}-p\varepsilon Q_b^{p-1}-\frac{p(p-1)}{2}\varepsilon^2Q_b^{p-2}\Big{]}y(Q_b)_y\bigg{\}}.
\end{split}\end{equation*}
We claim the following localized Virial estimate to obtain a coercivity result:
\begin{lemma}[Localized Virial estimate%
\footnote{See proof in \cite{MMR1} (Lemma 3.4 \& Lemma A.2).}%
] There exists $B_0>100$ and $\mu_1>0$ such that if $B>B_0$,  then:
\begin{multline*}
\int_{|y|<\kappa B}\big{(}3\varepsilon_y^2+\varepsilon^2-p\mathcal{Q}_p^{p-1}\varepsilon^2+p(p-1)y\mathcal{Q}_p'\mathcal{Q}_p^{p-2}\varepsilon^2\big{)}\\
\geq \mu_1\int_{|y|<\kappa B}(\varepsilon_y^2+\varepsilon^2)-\frac{1}{B}\int \varepsilon^2e^{-\frac{|y|}{2}}.
\end{multline*}
\end{lemma}
Since $\int_{|y|>\kappa B}\varepsilon^2e^{-\frac{|y|}{2}}\lesssim b_c^{10}$,
we have for some $\mu_2>0$:
$$\int_{|y|<\kappa B}\big{(}3\varepsilon_y^2+\varepsilon^2-p\mathcal{Q}_p^{p-1}\varepsilon^2+p(p-1)y\mathcal{Q}_p'\mathcal{Q}_p^{p-2}\varepsilon^2\big{)}\geq \mu_2\int_{|y|<\kappa B}(\varepsilon_y^2+\varepsilon^2)-b_c^{10}.$$
Using a similar strategy we have:
\begin{equation*}\begin{split}
|R(\varepsilon)|&\lesssim \frac{1}{B}\bigg{(}b_c\int_{|y|<\kappa B}\varepsilon^2+\int_{|y|<\kappa B}|\varepsilon|^3+|\varepsilon|^{p+1}\bigg{)}\\
&\lesssim \frac{1}{B}\bigg{(}(b_c+\|\varepsilon\|_{L^{\infty}})\int_{|y|<\kappa B}(\varepsilon_y^2+\varepsilon^2)\bigg{)}\\
&\lesssim \frac{1}{1000}\int_{|y|<\kappa B}(\varepsilon_y^2+\varepsilon^2)\varphi'_B.
\end{split}\end{equation*}
Collecting the above estimates, we obtain for some $\mu_3>0$:
\begin{equation}\label{54}
(f_{1,1})^{\sim}\leq -\mu_3\int_{|y|<\kappa B}(\varepsilon_y^2+\varepsilon^2)\varphi'_B+Cb_c^{\frac{7}{2}}.
\end{equation}
For the region $y<-\kappa B$, we have $\zeta_B(y)=\varphi_B(y)$ and $\psi_B\sim\varphi_B$. Hence, we immediately have:
\begin{align*}
&\int_{y<-\kappa B}\varepsilon^2|\zeta'''_B|\lesssim \frac{1}{B^2}\int_{y<-\kappa B}\varepsilon^2\varphi'_B\leq \frac{1}{100}\int_{y<-\kappa B}\varepsilon^2\varphi'_B,\\
&\int_{y<-\kappa B}\varepsilon_y^2|\psi'''_B|\lesssim \frac{1}{B^2}\int_{y<-\kappa B}\varepsilon_y^2\varphi'_B\leq \frac{1}{100}\int_{y<-\kappa B}\varepsilon_y^2\varphi'_B.
\end{align*}
From Lemma \ref{SSP}, we know that for $y<-\kappa B$, $|Q_b(y)|\lesssim b_c$ and $|Q_b'(y)|\lesssim b_c^2$. Recall that we have $\|\varepsilon\|_{L^{\infty}}\leq b_c^{\frac{1}{2}}$, then we can estimate:
\begin{equation*}\begin{split}
&\;\;\Bigg{|}\int_{y<-\kappa B}\bigg{[}\frac{|\varepsilon+Q_b|^{p+1}-Q_b^{p+1}}{p+1}-\varepsilon Q_b^p-\varepsilon\big{(}(\varepsilon+Q_b)|\varepsilon+Q_b|^{p-1}-Q_b^{p}\big{)}\bigg{]}(\zeta'_B-\psi'_B)\Bigg{|}\\
&\lesssim \int_{y<-\kappa B}\bigg{(}|\varepsilon|^{p+1}+|Q_b|^{p-1}\varepsilon^2\bigg{)}\varphi'_B\lesssim(b_c^{p-1}+\|\varepsilon\|^{p-1}_{L^{\infty}})\int_{y<-\kappa B}\varepsilon^2\varphi'_B\\
&\lesssim \frac{1}{100}\int_{y<-\kappa B}\varepsilon^2\varphi'_B,
\end{split}
\end{equation*}
\begin{equation*}\begin{split}
&\quad\Bigg{|}\int_{y<-\kappa B}\big{[}(\varepsilon+Q_b)|\varepsilon+Q_b|^{p-1}-Q_b^{p}-p\varepsilon Q_b^{p-1}\big{]}(Q_b)_y(\psi_B-\zeta_B)\Bigg{|}\\
&\lesssim B\int_{y<-\kappa B}(\varepsilon^2+|\varepsilon|^p)|(Q_b)_y|\varphi'_B\lesssim Bb_c^2(1+\|\varepsilon\|_{L^{\infty}}^{p-2})\int_{y<-\kappa B}\varepsilon^2\varphi'_B\\
&\lesssim\frac{1}{100}\int_{y<-\kappa B}\varepsilon^2\varphi'_B.
\end{split}\end{equation*}
Similarly, we have:
\begin{equation*}\begin{split}
&\quad\Bigg{|}\int_{y<-\kappa B}\psi'_B\varepsilon_y\big{\{}(Q_b)_y[|Q_b+\varepsilon|^{p-1}-Q_b^{p-1}]+|Q_b+\varepsilon|^{p-1}\varepsilon_y\big{\}}\Bigg{|}\\
&\lesssim \int_{y<-\kappa B}\Big{(}|\varepsilon_y\varepsilon(Q_b)_yQ_b^{p-2}|+|\varepsilon_y(Q_b)_y||\varepsilon|^{p-1}+|\varepsilon_y^2Q_b^{p-1}|+\varepsilon_y^2|\varepsilon|^{p-1}\Big{)}\varphi'_B\\
&\lesssim(b_c^{p-1}+\|\varepsilon\|^{p-1}_{L^{\infty}})\int_{y<-\kappa B}(\varepsilon_y^2+\varepsilon^2)\varphi'_B\\
&\lesssim\frac{1}{100}\int_{y<-\kappa B}(\varepsilon_y^2+\varepsilon^2)\varphi'_B
\end{split}\end{equation*}
and
\begin{equation*}\begin{split}
&\quad\Bigg{|}\int_{y<-\kappa B}\psi'_B\Big{\{}\big{[}-\varepsilon_{yy}+\varepsilon-\big{(}(\varepsilon+Q_b)|\varepsilon+Q_b|^{p-1}-Q_b^{p}\big{)}\big{]}^2-[-\varepsilon_{yy}+\varepsilon]^2\Big{\}}\Bigg{|}\\
&\lesssim \int_{y<-\kappa B}\Big{(}|\varepsilon\varepsilon_{yy }Q_b^{p-1}|+|\varepsilon_{yy}Q_b^{p}|+|\varepsilon^2 Q_b^{p-1}|+|\varepsilon Q_b^{p}|+|\varepsilon|^{2p}+\big{|}\varepsilon_{yy}|\varepsilon|^p\big{|}\Big{)}\psi'_B\\
&\lesssim \big{(}b_c^{p-1}+\|\varepsilon\|_{L^{\infty}}^{p-1}\big{)}\int_{y<-\kappa B}(\varepsilon_{yy}^2+\varepsilon^2)\psi'_B+\frac{1}{100}\int_{y<-\kappa B}(\varepsilon_{yy}^2+\varepsilon^2)\psi'_B\\
&\quad+100\int_{y<-\kappa B}Q_b^{2p}+\|\varepsilon\|_{L^{\infty}}^{2p-2}\int_{y<-\kappa B}\varepsilon^2\psi'_B\\
&\lesssim\frac{1}{100}\int_{y<-\kappa B}\big{(}\varepsilon_{yy}^2\psi'_B+\varepsilon^2\varphi'_B\big{)}+b_c^{\frac{7}{2}}.
\end{split}\end{equation*}
Therefore we obtain:
\begin{equation}\label{55}
(f_{1,1})^{<}\leq -\mu_4\int_{y<-\kappa B}(\varepsilon_y^2+\varepsilon^2)\varphi'_B+Cb_c^{\frac{7}{2}}
\end{equation}
for some $\mu_4>0$. From \eqref{53}, \eqref{54}, \eqref{55} and the following estimate:
\begin{equation*}\begin{split}
\int_{y>\kappa B}(\varepsilon_y^2+\varepsilon^2)\varphi'_B\lesssim& \frac{1}{B}\int_{y>\kappa B}\varepsilon_y^2+\frac{1}{B}\int_{\kappa B<y<2B^2}\varepsilon^2\\
\lesssim& b_c^4+B\|\varepsilon\|_{L^{\infty}(y>\kappa B)}^2\lesssim b_c^{\frac{7}{2}},
\end{split}\end{equation*}
we obtain for some $\mu_0>0$,
\begin{equation}\label{56}
f_{1,1}\leq -\mu_0\int(\varepsilon_y^2+\varepsilon^2)\varphi'_B+Cb_c^{\frac{7}{2}}.
\end{equation}
\subsubsection*{\underline{Term $f_{1,2}$}\normalfont{:}}We first rewrite $f_{1,2}$:
\begin{multline*}
f_{1,2}=2\bigg{(}\frac{\lambda_s}{\lambda}+b\bigg{)}\int(\Lambda(Q_b-\mathcal{Q}_p))\Big{\{}-(\psi_B\varepsilon_y)_y+\varepsilon\zeta_B\\
-\psi_B\big{[}(\varepsilon+Q_b)|\varepsilon+Q_b|^{p-1}-Q_b^{p}\big{]}\Big{\}}+\tilde{f}_{1,2},
\end{multline*}
where
\begin{align*}
\tilde{f}_{1,2}&=2\bigg{(}\frac{\lambda_s}{\lambda}+b\bigg{)}\int\Lambda\mathcal{Q}_p\Big{\{}-(\psi_B\varepsilon_y)_y+\varepsilon\zeta_B
-\psi_B\big{[}(\varepsilon+Q_b)|\varepsilon+Q_b|^{p-1}-Q_b^{p}\big{]}\Big{\}}\\
&=2\bigg{(}\frac{\lambda_s}{\lambda}+b\bigg{)}\int\Lambda \mathcal{Q}_p\big{[}-(\psi_B)_y\varepsilon_y+(1-\psi_B)\varepsilon_{yy}\big{]}\\
&\quad+2\bigg{(}\frac{\lambda_s}{\lambda}+b\bigg{)}\int\Lambda \mathcal{Q}_p\Big{\{}(1-\psi_B)\big{[}(\varepsilon+Q_b)|\varepsilon+Q_b|^{p-1}-Q_b^{p}\big{]}\Big{\}}\\
&\quad-2\bigg{(}\frac{\lambda_s}{\lambda}+b\bigg{)}\int\Lambda \mathcal{Q}_p\big{[}(\varepsilon+Q_b)|\varepsilon+Q_b|^{p-1}-Q_b^{p}-p\varepsilon \mathcal{Q}_p^{p-1}\big{]}+\tilde{\tilde{f}}_{1,2},\\
\tilde{\tilde{f}}_{1,2}&=2\bigg{(}\frac{\lambda_s}{\lambda}+b\bigg{)}\int\Lambda\mathcal{Q}_p\big{(}-\varepsilon_{yy}+\varepsilon\zeta_B-p\mathcal{Q}_p^{p-1}\varepsilon\big{)}\\
&=2\bigg{(}\frac{\lambda_s}{\lambda}+b\bigg{)}\int\Lambda \mathcal{Q}_p(L\varepsilon)-2\bigg{(}\frac{\lambda_s}{\lambda}+b\bigg{)}\int\varepsilon(1-\zeta_B)\Lambda \mathcal{Q}_p.
\end{align*}
In conclusion, we have:
\begin{equation*}\begin{split}
f_{1,2}&=2\bigg{(}\frac{\lambda_s}{\lambda}+b\bigg{)}\int\Lambda \mathcal{Q}_p(L\varepsilon)-2\bigg{(}\frac{\lambda_s}{\lambda}+b\bigg{)}\int\varepsilon(1-\zeta_B)\Lambda \mathcal{Q}_p\\
&+2\bigg{(}\frac{\lambda_s}{\lambda}+b\bigg{)}\int(\Lambda(Q_b-\mathcal{Q}_p))\Big{\{}-(\psi_B\varepsilon_y)_y+\varepsilon\zeta_B\\
&\qquad\qquad\qquad\quad-\psi_B\big{[}(\varepsilon+Q_b)|\varepsilon+Q_b|^{p-1}-Q_b^{p}\big{]}\Big{\}}\\
&+2\bigg{(}\frac{\lambda_s}{\lambda}+b\bigg{)}\int\Lambda \mathcal{Q}_p\big{[}-(\psi_B)_y\varepsilon_y+(1-\psi_B)\varepsilon_{yy}\big{]}\\
&+2\bigg{(}\frac{\lambda_s}{\lambda}+b\bigg{)}\int\Lambda \mathcal{Q}_p\Big{\{}(1-\psi_B)\big{[}(\varepsilon+Q_b)|\varepsilon+Q_b|^{p-1}-Q_b^{p}\big{]}\Big{\}}\\
&-2\bigg{(}\frac{\lambda_s}{\lambda}+b\bigg{)}\int\Lambda \mathcal{Q}_p\big{[}(\varepsilon+Q_b)|\varepsilon+Q_b|^{p-1}-Q_b^{p}-p\varepsilon \mathcal{Q}_p^{p-1}\big{]}.
\end{split}\end{equation*}

We know from the orthogonality condition \eqref{OC} that:
$$\int\Lambda \mathcal{Q}_p(L\varepsilon)=(\varepsilon,L\Lambda \mathcal{Q}_p)=-2(\varepsilon, \mathcal{Q}_p)=0.$$
Again from the orthogonality condition $(\varepsilon,y\Lambda \mathcal{Q}_p)=0$, we can estimate:
\begin{align*}
\Bigg{|}\int\Lambda \mathcal{Q}_p\varepsilon(1-\zeta_B)\Bigg{|}&=\Bigg{|}\int\Lambda \mathcal{Q}_p \varepsilon\bigg{(}1-\zeta_B+\frac{y}{B}\bigg{)})\Bigg{|}\\
&\lesssim e^{-\frac{\kappa B}{20}}\|\varepsilon\|_{L^{\infty}}\leq b_c^{\frac{7}{2}}.
\end{align*}

For the next term, we first integrate by parts to remove all the derivatives on $\varepsilon$, then we divide the integral into 2 parts, $\int_{y<\kappa B}$ and $\int_{y>\kappa B}$. For the first part we use Cauchy-Schwarz inequality, \eqref{ASB} and \eqref{ASB2}. While for the second part we use the fact that $Q_b$ decays exponentially on the right. So we have:
\begin{equation*}\begin{split}
&\quad\;\Bigg{|}\bigg{(}\frac{\lambda_s}{\lambda}+b\bigg{)}\int\Lambda(Q_b-\mathcal{Q}_p)\big{\{}-\psi_B\big{[}(\varepsilon+Q_b)|\varepsilon+Q_b|^{p-1}-Q_b^{p}\big{]}-(\psi_B\varepsilon_y)_y+\varepsilon\zeta_B\big{\}}\Bigg{|}\\
&=\Bigg{|}\bigg{(}\frac{\lambda_s}{\lambda}+b\bigg{)}\int\bigg{(}(\Lambda Q_b-\Lambda \mathcal{Q}_p)\big{\{}-\psi_B\big{[}(\varepsilon+Q_b)|\varepsilon+Q_b|^{p-1}-Q_b^{p}\big{]}+\varepsilon\zeta_B\big{\}}\\
&\qquad\qquad\qquad\quad-\Big{[}(\Lambda Q_b-\Lambda \mathcal{Q}_p)'\psi_B\Big{]}'\varepsilon\bigg{)}\Bigg{|}\\
&\lesssim\bigg{(}b_c^{\frac{5}{2}}+\mathcal{N}^{\frac{1}{2}}\bigg{)}\bigg{(}b_c\int_{y<\kappa B}\psi_B(|\varepsilon|+|\varepsilon|^p)+\int_{y>\kappa B}e^{-\frac{|y|}{10}}(|\varepsilon|+|\varepsilon|^p)\bigg{)}\\
&\lesssim \bigg{(}b_c^{\frac{5}{2}}+\mathcal{N}^{\frac{1}{2}}\bigg{)}\bigg{(}b_c\bigg{(}\int_{y<\kappa B}\varepsilon^2\psi_B\bigg{)}^{\frac{1}{2}}\bigg{(}\int_{y<\kappa B}\psi_B\bigg{)}^{\frac{1}{2}}+e^{-\frac{\kappa B}{20}}\|\varepsilon\|_{L^{\infty}}\bigg{)}\\
&\lesssim b_cB^{\frac{3}{2}}\int(\varepsilon_y^2+\varepsilon^2)\varphi'_B+b_c^{\frac{7}{2}}B\bigg{(}\int(\varepsilon_y^2+\varepsilon^2)\varphi'_B\bigg{)}^{\frac{1}{2}}+b_c^{\frac{7}{2}}\\
&\leq \frac{\mu_0}{1000}\int(\varepsilon_y^2+\varepsilon^2)\varphi'_B+Cb_c^{\frac{7}{2}}.
\end{split}\end{equation*}

For the following 2 terms, we first integrate by parts again to remove the derivatives on $\varepsilon$.  Then we use the fact that $\psi_B=1$ on $[-\kappa B,+\infty)$ and $$|(\Lambda \mathcal{Q}_p)''(y)|+|\Lambda \mathcal{Q}_p(y)|\lesssim e^{-\frac{\kappa B}{20}}\varphi'_B(y)$$ for $y<-\kappa B$, to obtain:
\begin{equation*}\begin{split}
&\quad\;\Bigg{|}\bigg{(}\frac{\lambda_s}{\lambda}+b\bigg{)}\int\Lambda \mathcal{Q}_p\big{[}-(\psi_B)_y\varepsilon_y+(1-\psi_B)\varepsilon_{yy}\big{]}\Bigg{|}\\
&=\Bigg{|}\bigg{(}\frac{\lambda_s}{\lambda}+b\bigg{)}\int\Big{\{}\big{[}\Lambda \mathcal{Q}_p(1-\psi_B)\big{]}''\varepsilon+\big{[}\Lambda \mathcal{Q}_p(\psi_B)_y\big{]}'\varepsilon\Big{\}}\Bigg{|}\\
&\lesssim\bigg{(}b_c^{\frac{5}{2}}+\mathcal{N}^{\frac{1}{2}}\bigg{)}\bigg{(}\int_{y<-\kappa B}\varepsilon^2\varphi'_B\bigg{)}^{\frac{1}{2}}e^{-\frac{\kappa B}{30}}\\
&\leq\frac{\mu_0}{1000}\int(\varepsilon_y^2+\varepsilon^2)\varphi'_B+Cb_c^{\frac{7}{2}}
\end{split}\end{equation*}
and
\begin{align*}
&\quad\;\Bigg{|}\bigg{(}\frac{\lambda_s}{\lambda}+b\bigg{)}\int\Lambda \mathcal{Q}_p\Big{\{}(1-\psi_B)\big{[}(\varepsilon+Q_b)|\varepsilon+Q_b|^{p-1}-Q_b^{p}\big{]}\Big{\}}\Bigg{|}\\
&\lesssim\bigg{(}b_c^{\frac{5}{2}}+\mathcal{N}^{\frac{1}{2}}\bigg{)}\int_{y<-\kappa B}\big{(}|\varepsilon|+|\varepsilon|^{p}\big{)}e^{-\frac{\kappa B}{20}}\varphi'_B\\
&\lesssim\bigg{(}b_c^{\frac{5}{2}}+\mathcal{N}^{\frac{1}{2}}\bigg{)}\bigg{(}\int_{y<-\kappa B}\varepsilon^2\varphi'_B\bigg{)}^{\frac{1}{2}}e^{-\frac{\kappa B}{20}}\\
&\leq\frac{\mu_0}{1000}\int(\varepsilon_y^2+\varepsilon^2)\varphi'_B+Cb_c^{\frac{7}{2}}.
\end{align*}

Finally, by the same strategy we have :
\begin{equation*}\begin{split}
&\quad\;\Bigg{|}\bigg{(}\frac{\lambda_s}{\lambda}+b\bigg{)}\int\Lambda \mathcal{Q}_p\big{[}(\varepsilon+Q_b)|\varepsilon+Q_b|^{p-1}-Q_b^{p}-p\varepsilon \mathcal{Q}_p^{p-1}\big{]}\Bigg{|}\\
&=\Bigg{|}\bigg{(}\frac{\lambda_s}{\lambda}+b\bigg{)}\int\Lambda \mathcal{Q}_p\Big{[}(\varepsilon+Q_b)|\varepsilon+Q_b|^{p-1}-Q_b^{p}-p\varepsilon Q_b^{p-1}+p\varepsilon\big{(}Q_b^{p-1}-\mathcal{Q}_p^{p-1}\big{)}\Big{]}\Bigg{|}\\
&\lesssim \bigg{(}b_c^{\frac{5}{2}}+\mathcal{N}^{\frac{1}{2}}\bigg{)}\bigg{(}\int_{y<\kappa B}\varepsilon^2\varphi'_B+b_c\mathcal{N}^{\frac{1}{2}}+e^{-\frac{\kappa B}{20}}\|\varepsilon\|_{L^{\infty}}\bigg{)}\\
&\leq\frac{\mu_0}{1000}\int(\varepsilon_y^2+\varepsilon^2)\varphi'_B+Cb_c^{\frac{7}{2}}.
\end{split}\end{equation*}

The collection of the above estimates shows that:
\begin{equation}\label{57}
|f_{1,2}|\leq\frac{\mu_0}{100}\int(\varepsilon_y^2+\varepsilon^2)\varphi'_B+Cb_c^{\frac{7}{2}}.
\end{equation}
\subsubsection*{\underline{Term $f_{1,3}$}\normalfont{:}}
We use the identity:
\begin{equation*}\begin{split}
&\quad\int\psi_B\Big{\{}(Q_b)_y\big{[}(\varepsilon+Q_b)|\varepsilon+Q_b|^{p-1}-Q_b^{p}-p\varepsilon Q_b^{p-1}\big{]}\\
&\qquad\quad\quad+\varepsilon_y\big{[}(\varepsilon+Q_b)|\varepsilon+Q_b|^{p-1}-Q_b^{p}\big{]}\Big{\}}\\
&=\frac{1}{p+1}\int\psi_B\partial_y\big{[}|Q_b+\varepsilon|^{p+1}-Q_b^{p+1}-(p+1)\varepsilon Q_b^p\big{]}\\
&=-\frac{1}{p+1}\int\psi'_B\big{[}|Q_b+\varepsilon|^{p+1}-Q_b^{p+1}-(p+1)\varepsilon Q_b^p\big{]}
\end{split}\end{equation*}
and a similar computation (as we do for term $f_{1,2}$) to rewrite $f_{1,3}$:
\begin{equation*}\begin{split}
f_{1,3}&=\frac{2}{p+1}\bigg{(}\frac{x_s}{\lambda}-1\bigg{)}\int\psi'_B\big{[}|Q_b+\varepsilon|^{p+1}-Q_b^{p+1}-(p+1)\varepsilon Q_b^p\big{]}\\
&\quad+2\bigg{(}\frac{x_s}{\lambda}-1\bigg{)}\int(Q_b-\mathcal{Q}_p+\varepsilon)_y\big{(}-\psi'_B\varepsilon_y-\psi_B\varepsilon_{yy}+\varepsilon\zeta_B\big{)}\\
&\quad-2p\bigg{(}\frac{x_s}{\lambda}-1\bigg{)}\int \varepsilon\psi_B\big{[}Q_b^{p-1}(Q_b)_y-\mathcal{Q}_p^{p-1}(\mathcal{Q}_p)_y\big{]}\\
&\quad+2\bigg{(}\frac{x_s}{\lambda}-1\bigg{)}\int \mathcal{Q}_p'\big{[}L\varepsilon-\psi'_B\varepsilon_y+(1-\psi_B)\varepsilon_{yy}-\varepsilon(1-\zeta_B)\big{]}.
\end{split}\end{equation*}

For the first term, we use the bootstrap assumption $\mathcal{N}\leq b_c^3$ to estimate:
\begin{equation*}\begin{split}
&\quad\;\Bigg{|}\bigg{(}\frac{x_s}{\lambda}-1\bigg{)}\int\psi'_B\big{[}|Q_b+\varepsilon|^{p+1}-Q_b^{p+1}-(p+1)\varepsilon Q_b^p\big{]}\Bigg{|}\\
&\lesssim\bigg{(}b_c^{\frac{5}{2}}+\mathcal{N}^{\frac{1}{2}}\bigg{)}\int\psi'_B(|\varepsilon|^{p+1}+\varepsilon^2Q_b^{p-1})\\
&\leq \frac{\mu_0}{1000}\int(\varepsilon_y^2+\varepsilon^2)\varphi'_B.
\end{split}\end{equation*}

For the second term, we first integrate by parts to remove the derivatives of $\varepsilon$, then we use Cauchy-Schwarz inequality, \eqref{ASB} and \eqref{ASB2} to estimate  $\int_{y<\kappa B}$ and use \eqref{49} to estimate $\int_{y>\kappa B}$ as before:
\begin{align*}
&\quad\;\Bigg{|}\bigg{(}\frac{x_s}{\lambda}-1\bigg{)}\int(Q_b-\mathcal{Q}_p)_y\big{(}-\psi'_B\varepsilon_y-\psi_B\varepsilon_{yy}+\varepsilon\zeta_B\big{)}\Bigg{|}\\
&\lesssim \bigg{(}b_c^{\frac{5}{2}}+\mathcal{N}^{\frac{1}{2}}\bigg{)}\bigg{(}b_cB\mathcal{N}^{\frac{1}{2}}+e^{-\frac{\kappa B}{20}}\|\varepsilon\|_{L^{\infty}}\bigg{)}\\
&\leq \frac{\mu_0}{1000}\int(\varepsilon_y^2+\varepsilon^2)\varphi'_B+Cb_c^{\frac{7}{2}},\\
&\quad\;\Bigg{|}\bigg{(}\frac{x_s}{\lambda}-1\bigg{)}\int\varepsilon_y\big{(}-\psi'_B\varepsilon_y-\psi_B\varepsilon_{yy}+\varepsilon\zeta_B\big{)}\Bigg{|}\\
&\lesssim\bigg{(}b_c^{\frac{5}{2}}+\mathcal{N}^{\frac{1}{2}}\bigg{)}\bigg{(}\int(\varepsilon_y^2+\varepsilon^2)\varphi'_B+\frac{1}{B^2}\int_{B^2<y<2B^2}\varepsilon^2\bigg{)}\\
&\leq\frac{\mu_0}{1000}\int(\varepsilon_y^2+\varepsilon^2)\varphi'_B+Cb_c^{\frac{7}{2}}.
\end{align*}

For the next term, we can estimate similarly by dividing the integral into 2 parts:
\begin{align*}
&\quad\;\Bigg{|}\bigg{(}\frac{x_s}{\lambda}-1\bigg{)}\int \varepsilon\psi_B\big{[}Q_b^{p-1}(Q_b)_y-\mathcal{Q}_p^{p-1}(\mathcal{Q}_p)_y\big{]}\Bigg{|}\\
&\lesssim \bigg{(}b_c^{\frac{5}{2}}+\mathcal{N}^{\frac{1}{2}}\bigg{)}\big{(}b_cB\mathcal{N}^{\frac{1}{2}}+e^{-\frac{\kappa B}{20}}\|\varepsilon\|_{L^{\infty}}\big{)}\\
&\leq \frac{\mu_0}{1000}\int(\varepsilon_y^2+\varepsilon^2)\varphi'_B+Cb_c^{\frac{7}{2}}.
\end{align*}

For the last term, we use the cancellation $L\mathcal{Q}_p'$=0 and the orthogonality condition $(\varepsilon,y\mathcal{Q}_p')=(\varepsilon,\Lambda \mathcal{Q}_p-\frac{2}{p-1}\mathcal{Q}_p)=0$ to estimate:
\begin{equation*}\begin{split}
&\quad\;\Bigg{|}\bigg{(}\frac{x_s}{\lambda}-1\bigg{)}\int \mathcal{Q}_p'\big{[}L\varepsilon-\psi'_B\varepsilon_y+(1-\psi_B)\varepsilon_{yy}-\varepsilon(1-\zeta_B)\big{]}\Bigg{|}\\
&=\Bigg{|}\bigg{(}\frac{x_s}{\lambda}-1\bigg{)}\int \mathcal{Q}_p'\big{[}L\varepsilon-\psi'_B\varepsilon_y+(1-\psi_B)\varepsilon_{yy}-\varepsilon(1+\frac{y}{B}-\zeta_B)\big{]}\Bigg{|}\\
&\lesssim \bigg{(}b_c^{\frac{5}{2}}+\mathcal{N}^{\frac{1}{2}}\bigg{)}\big{(}e^{-\frac{\kappa B}{20}}\mathcal{N}^{\frac{1}{2}}+e^{-\frac{\kappa B}{20}}\|\varepsilon\|_{L^{\infty}}\big{)}\\
&\leq \frac{\mu_0}{1000}\int(\varepsilon_y^2+\varepsilon^2)\varphi'_B+Cb_c^{\frac{7}{2}}.
\end{split}\end{equation*}

In conclusion, we have:
\begin{equation}\label{58}
|f_{1,3}|\leq\frac{\mu_0}{100}\int(\varepsilon_y^2+\varepsilon^2)\varphi'_B+Cb_c^{\frac{7}{2}}.
\end{equation}
\subsubsection*{\underline{Term $f_{1,4}$}\normalfont{:}}Recall that
$$f_{1,4}=-2b_s\int P_b\Big{\{}-(\psi_B\varepsilon_y)_y+\varepsilon\zeta_B-\psi_B\big{[}(\varepsilon+Q_b)|\varepsilon+Q_b|^{p-1}-Q_b^{p}\big{]}\Big{\}}.$$
We estimate after integration by parts to remove the derivatives of $\varepsilon$ and then divide the integral into 2 parts as before:
\begin{align*}
&\quad\;\Bigg{|}\int P_b(-(\psi_B\varepsilon_y)_y+\varepsilon\zeta_B)\Bigg{|}=\Bigg{|}\int\big{(}(P_b)_y\varepsilon_y\psi_B+\varepsilon P_b\zeta_B\big{)}\Bigg{|}\\
&\lesssim\Bigg{|}\int_{y<\kappa B}\big{(}|\varepsilon P_b|\zeta_B+|\varepsilon_y (P_b)_y|\psi_B\big{)}+\int_{y>\kappa B} e^{-\frac{y}{8}}\big{(}|\varepsilon|+|\varepsilon_y|\big{)}\Bigg{|}\\
&\lesssim B\bigg{(}\int_{y<\kappa B}(\varepsilon_y^2+\varepsilon^2)\varphi'_B\bigg{)}^{\frac{1}{2}}+e^{-\frac{\kappa B}{20}}\bigg{(}\int_{y>\kappa B}\varepsilon_y^2+\|\varepsilon\|_{L^{\infty}}^2\bigg{)}^{\frac{1}{2}}\\
&\lesssim B\mathcal{N}^{\frac{1}{2}}+b_c^{\frac{7}{2}}.
\end{align*}
For the nonlinear term, the same strategy shows:
\begin{align*}
&\quad\;\Bigg{|}\int P_b\psi_B\big{[}(\varepsilon+Q_b)|\varepsilon+Q_b|^{p-1}-Q_b^{p}\big{]}\Bigg{|}\lesssim\int|P_b|\psi_B(Q_b^{p-1}|\varepsilon|+|\varepsilon|^p)\\
&\lesssim B\bigg{(}\int_{y<\kappa B}\varepsilon^2\varphi'_B\bigg{)}^{\frac{1}{2}}+e^{-\frac{\kappa B}{20}}\|\varepsilon\|_{L^{\infty}}\\
&\lesssim B\mathcal{N}^{\frac{1}{2}}+b_c^{\frac{7}{2}}.
\end{align*}
Recall from \eqref{MES3} we have:
$$|b_s|\lesssim b_c^{\frac{5}{2}}+b_c\mathcal{N}^{\frac{1}{2}}.$$
Then we obtain:
\begin{equation}\label{59}
|f_{1,4}|\leq \frac{\mu_0}{100}\int(\varepsilon_y^2+\varepsilon^2)\varphi'_B+Cb_c^{\frac{7}{2}}.
\end{equation}
\subsubsection*{\underline{Term $f_{1,5}$}\normalfont{:}}
Recall from \eqref{APP} we have for $k=0,1$:
$$|\partial_y^k\Phi_b|\lesssim b_c|\tilde{b}||\partial_y^kQ_b|+b_c^2\mathbf{1}_{[-2,-1]}(b_cy)+e^{-\frac{1}{10b_c}}\mathbf{1}_{[1,2]}(b_cy)$$
So after integration by parts, we have:
\begin{align*}
&\quad\;\Bigg{|}\int\Phi_b(-(\psi_B\varepsilon_y)_y+\varepsilon\zeta_B)\Bigg{|}=\Bigg{|}\int(\Phi_b)_y\psi_B\varepsilon_y+\int\Phi_b\varepsilon\zeta_B\Bigg{|}\\
&\lesssim b_c^{\frac{5}{2}}\int \big{(}Q_b+|\partial_yQ_b|\big{)}\big{(}|\varepsilon_y\psi_B|+|\varepsilon\zeta_B|\big{)}+b_c^2\int_{y\sim -b_c^{-1}}\big{(}|\varepsilon_y\psi_B|+|\varepsilon\zeta_B|\big{)}\\
&\quad+e^{-\frac{1}{10b_c}}\int_{y\sim b_c^{-1}}|\varepsilon_y\psi_B|+|\varepsilon\zeta_B|\\
&\lesssim b_c^{\frac{5}{2}}B\mathcal{N}^{\frac{1}{2}}+e^{-\frac{1}{2Bb_c}}(\|\varepsilon\|_{L^{\infty}}+\|\varepsilon_y\|_{L^2})\\
&\leq\frac{\mu_0}{1000}\int(\varepsilon_y^2+\varepsilon^2)\varphi'_B+Cb_c^{\frac{7}{2}}.
\end{align*}
Here we use the fact that $|\psi_B(y)|+|\zeta_B(y)|\lesssim e^{-\frac{1}{2Bb_c}}\lesssim b_c^{10}$, for all $y\in[-2b_c^{-1},-b_c^{-1}]$.

The nonlinear term can be similarly estimated as before:
\begin{align*}
&\quad\;\Bigg{|}\int\Phi_b\big{[}(\varepsilon+Q_b)|\varepsilon+Q_b|^{p-1}-Q_b^{p}\big{]}\psi_B\Bigg{|}\lesssim \int|\Phi_b|\psi_B\big{(}|\varepsilon|^p+|Q_b^{p-1}\varepsilon|\big{)}\\
&\lesssim b_c^2\int_{y\sim -b_c^{-1}}\big{(}|\varepsilon Q_b^{p-1}|+|\varepsilon|^p\big{)}\psi_B+e^{-\frac{1}{10b_c}}\int_{y\sim b_c^{-1}}\big{(}|\varepsilon Q_b^{p-1}|+|\varepsilon|^p\big{)}\psi_B\\
&\quad+b_c^{\frac{5}{2}}\int Q_b\big{(}|\varepsilon Q_b^{p-1}|+|\varepsilon|^p\big{)}\psi_B\\
&\leq\frac{\mu_0}{1000}\int(\varepsilon_y^2+\varepsilon^2)\varphi'_B+Cb_c^{\frac{7}{2}}.
\end{align*}
Thus we have shown that:
\begin{equation}\label{510}
|f_{1,5}|\leq \frac{\mu_0}{100}\int(\varepsilon_y^2+\varepsilon^2)\varphi'_B+Cb_c^{\frac{7}{2}}.
\end{equation}
\subsection*{Step 3 \normalfont{Control of $f_2$}}
Recall that:
$$f_2=2\frac{\lambda_s}{\lambda}\int\Lambda\varepsilon\Big{\{}-(\psi_B\varepsilon_y)_y+\varepsilon\zeta_B-\psi_B\big{[}(\varepsilon+Q_b)|\varepsilon+Q_b|^{p-1}-Q_b^{p}\big{]}\Big{\}}.$$
We first claim the following identities:
\begin{align}
&\int\Lambda \varepsilon(\psi_B\varepsilon_y)_y=-(1-\sigma_c)\int\varepsilon_y^2\psi_B+\frac{1}{2}\int y\psi'_B\varepsilon_y^2,\label{50001}\\
&\int\Lambda \varepsilon(\varepsilon\zeta_B)=-\sigma_c\int\varepsilon^2\zeta_B-\frac{1}{2}\int y\zeta'_B\varepsilon^2,\label{50002}\\
&\int\Lambda \varepsilon\psi_B\big{[}(\varepsilon+Q_b)|\varepsilon+Q_b|^{p-1}-Q_b^{p}\big{]}\nonumber\\
&=\frac{1}{p+1}\int\bigg{(}\frac{p+3}{p-1}\psi_B-y\psi'_B\bigg{)}\big{[}|Q_b+\varepsilon|^{p+1}-Q_b^{p+1}-(p+1)\varepsilon Q_b^p\big{]}\label{50003}\\
&\quad-\int\psi_B\Lambda Q_b\big{[}(\varepsilon+Q_b)|\varepsilon+Q_b|^{p-1}-Q_b^{p}-p\varepsilon Q_b^{p-1}\big{]}.\nonumber
\end{align}
We can see \eqref{50001} and \eqref{50002} are easily obtained by integrating by parts. While for \eqref{50003}, we have the following computation:
\begin{align*}
&\quad\;\int\Lambda(\varepsilon+Q_b)\psi_B\big{[}(\varepsilon+Q_b)|\varepsilon+Q_b|^{p-1}-Q_b^{p}\big{]}\\
&=\int\frac{2}{p-1}\psi_B\big{[}|\varepsilon+Q_b|^{p+1}-Q_b^{p+1}-\varepsilon Q_b^p\big{]}\\
&\quad+\int y(\varepsilon+Q_b)'\psi_B\big{[}(\varepsilon+Q_b)|\varepsilon+Q_b|^{p-1}-Q_b^{p}\big{]}\\
&=\frac{2}{p-1}\int\psi_B\big{[}|\varepsilon+Q_b|^{p+1}-Q_b^{p+1}-(p+1)\varepsilon Q_b^p\big{]}+p\int\psi_B\varepsilon Q_b^{p-1}(\frac{2}{p-1}Q_b)+\Delta,
\end{align*}
with $$\Delta=\int y(\varepsilon+Q_b)'\psi_B\big{[}(\varepsilon+Q_b)|\varepsilon+Q_b|^{p-1}-Q_b^{p}\big{]}.$$
Then we use the following identity:
\begin{multline*}
\big{[}|\varepsilon+Q_b|^{p+1}-Q_b^{p+1}-(p+1)\varepsilon Q_b^p\big{]}'\\=(p+1)(\varepsilon+Q_b)'\big{[}(\varepsilon+Q_b)|\varepsilon+Q_b|^{p-1}-Q_b^{p}\big{]}-p(p+1)\varepsilon Q_b'Q_b^{p-1}
\end{multline*} 
to compute:
\begin{align*}
\Delta&=\frac{1}{p+1}\int y\psi_B \big{[}|\varepsilon+Q_b|^{p+1}-Q_b^{p+1}-(p+1)\varepsilon Q_b^p\big{]}'-p\int y\psi_B \varepsilon Q_b'Q_b^{p-1}\\
&=-\frac{1}{p+1}\int (\psi_B-y\psi_B') \big{[}|\varepsilon+Q_b|^{p+1}-Q_b^{p+1}-(p+1)\varepsilon Q_b^p\big{]}\\
&\quad-p\int \psi_B \varepsilon Q_b^{p-1}(yQ_b').
\end{align*}
Collecting all the above computation, we have:
\begin{align*}
&\quad\;\int\Lambda(\varepsilon+Q_b)\psi_B\big{[}(\varepsilon+Q_b)|\varepsilon+Q_b|^{p-1}-Q_b^{p}\big{]}\\
&=\frac{1}{p+1}\int\bigg{(}\frac{p+3}{p-1}\psi_B-y\psi'_B\bigg{)}\big{[}|Q_b+\varepsilon|^{p+1}-Q_b^{p+1}-(p+1)\varepsilon Q_b^p\big{]}\\
&\quad+p\int \psi_B\varepsilon Q_b^{p-1}(\Lambda Q_b),
\end{align*}
which is just \eqref{50003}.

Now we can use \eqref{50001}--\eqref{50003} to estimate $f_{1,2}$. Since
$$\frac{\lambda_s}{\lambda}\sim-b_c<0,$$
we can drop the negative term to obtain:
\begin{equation*}
2\frac{\lambda_s}{\lambda}\int\Lambda \varepsilon(-\psi_B\varepsilon)_y\leq 0
\end{equation*}
and
\begin{align*}
2\frac{\lambda_s}{\lambda}\int\Lambda \varepsilon\zeta_B&\leq C\bigg{(}b_c^2\int\varepsilon^2\zeta_B+b_c\int_{0<y<B}y\varphi'_B\varepsilon^2+b_c\int y\eta'_B\varepsilon^2\bigg{)}\\
&\leq C\bigg{(}b_cB\int_{y<\kappa B}\varepsilon^2\varphi'_B+b_c^{\frac{1}{2}}\int_{\kappa B<y<2B^2}\varepsilon^2\bigg{)}\\
&\leq C\bigg{(}b_c^{\frac{1}{2}}\int_{y<\kappa B}\varepsilon^2\varphi'_B+\|\varepsilon\|_{L^{\infty}(y>\kappa B)}^2\bigg{)}\\
&\leq \frac{\mu_0}{1000}\int(\varepsilon_y^2+\varepsilon^2)\varphi'_B+Cb_c^{\frac{7}{2}}.
\end{align*}
For the nonlinear term we divide the integral into 3 parts:
$$\int\Lambda \varepsilon\psi_B\big{[}(\varepsilon+Q_b)|\varepsilon+Q_b|^{p-1}-Q_b^{p}\big{]}=m^{<}+m^{\sim}+m^{>},$$
where $m^{<}$, $m^{\sim}$ and $m^{>}$ correspond to the integration on $y<-\kappa B$, $|y|<\kappa B$ and $y>\kappa B$ respectively.
For $y>\kappa B$, we have:
\begin{align*}
|m^{>}|\lesssim \int_{y>\kappa B}\big{(}|\varepsilon|^{p+1}+\varepsilon^2 e^{-\frac{|y|}{2}}\big{)}\leq b_c^{\frac{7}{2}}.
\end{align*}
Next for $|y|<\kappa B$, we can estimate:
$$|m^{\sim}|\lesssim \int_{|y|<\kappa B}|\varepsilon|^{p+1}+\varepsilon^2\lesssim B\int(\varepsilon_y^2+\varepsilon^2)\varphi'_B.$$
Finally, for $y<-\kappa B$, we have $|Q_b|+|\Lambda Q_b|\lesssim b_c$ on this region. Together with $\|\varepsilon\|_{L^{\infty}}\leq b_c^{\frac{1}{2}}$, we obtain:
\begin{align*}
|m^{<}|&\lesssim \Big{(}\|\varepsilon\|_{L^{\infty}}^{p+1}+b_c^{p-1}\|\varepsilon\|_{L^{\infty}}^2+b_c\|\varepsilon\|_{L^{\infty}}^{p}\Big{)}\int_{y<-\kappa B}\big{(}|\psi_B|+|y\psi'_B|\big{)}\\
&\lesssim Bb_c^{3}\leq b_c^{\frac{5}{2}}.
\end{align*}
Therefore, we obtain:
$$\Bigg{|}\frac{\lambda_s}{\lambda}\int\Lambda \varepsilon\psi_B\big{[}(\varepsilon+Q_b)|\varepsilon+Q_b|^{p-1}-Q_b^{p}\big{]}\Bigg{|}\lesssim\frac{\mu_0}{1000}\int(\varepsilon_y^2+\varepsilon^2)\varphi'_B+b_c^{\frac{7}{2}},$$
hence
\begin{equation}\label{511}
f_{2}\leq \frac{\mu_0}{100}\int(\varepsilon_y^2+\varepsilon^2)\varphi'_B+Cb_c^{\frac{7}{2}}.
\end{equation}
\subsection*{Step 4 \normalfont{Control of $f_3$}}
First from (3.8)
$$|(Q_b)_s|=|b_sP_b|\lesssim b_c^{\frac{5}{2}}|P_b|.$$
Recalling that $P_b$ decays exponentially on the right, we have:
\begin{equation}\label{512}\begin{split}
|f_3|&\lesssim b_c^{\frac{5}{2}}\bigg{(}\int_{y<\kappa B}\psi_B(|\varepsilon|^p+\varepsilon^2)+e^{-\frac{\kappa B}{20}}\|\varepsilon\|_{L^{\infty}}^2\bigg{)}\\
&\leq \frac{\mu_0}{100}\int(\varepsilon_y^2+\varepsilon^2)\varphi'_B+Cb_c^{\frac{7}{2}}.
\end{split}\end{equation}
Collecting \eqref{56}--\eqref{512}, we conclude the proof of \eqref{51} and \eqref{52}.
\subsection*{Step 5 \normalfont{Coercivity of $\mathcal{F}$}}
As before we divide the integral into 2 parts, $\mathcal{F}^{<}$ and $\mathcal{F}^{>}$, which correspond to the integration on $y<\kappa B$ and $y>\kappa B$ respectively.

For the upper bound of $\mathcal{F}$, recall that $B=b_c^{-\frac{1}{20}}$, we have for $y>\kappa B$,
\begin{align*}
|\mathcal{F}^{>}|&\lesssim \int_{y>\kappa B}\big{(}\varepsilon_y^2+|\varepsilon|^{p+1}+\varepsilon^2 e^{-\frac{|y|}{2}}\big{)}+\int_{\kappa B<y<2B^2}\varepsilon^2\\
&\lesssim b_c^8+B^2\|\varepsilon\|_{L^{\infty}(y>\kappa B)}^2\lesssim b_c^8+b_c^{-\frac{1}{10}+4}\\
&\leq b_c^{\frac{7}{2}}.
\end{align*}
And for $y<\kappa B$, we have:
\begin{align*}
|\mathcal{F}^{<}|&\lesssim \int_{y<\kappa B}\big{(}\varepsilon_y^2+\varepsilon^2+|\varepsilon|^{p+1}\big{)}\psi_B\\
&\lesssim B\int_{y<\kappa B}(\varepsilon_y^2+\varepsilon^2)\varphi'_B\leq \mathcal{N}.
\end{align*}
Then the upper bound follows.

For the lower bound, we rewrite $\mathcal{F}$:
\begin{align*}
\mathcal{F}=&\int\big{(}\varepsilon_y^2\psi_B+\varepsilon^2\zeta_B-p\psi_B\mathcal{Q}_p^{p-1}\varepsilon^2\big{)}-p\int\psi_B(Q_b^{p-1}-\mathcal{Q}_p^{p-1})\varepsilon^2\\
&-\frac{2}{p+1}\int\psi_B\Big{[}|Q_b+\varepsilon|^{p+1}-Q_b^{p+1}-(p+1)\varepsilon Q_b^{p}-\frac{p(p+1)}{2}Q_b^{p-1}\varepsilon^2\Big{]}.
\end{align*}
First, we have:
\begin{align*}
\quad\Bigg{|}\int\psi_B(Q_b^{p-1}-\mathcal{Q}_p^{p-1})\varepsilon^2\Bigg{|}\lesssim b_cB\int_{y<\kappa B}\varepsilon^2\varphi'_B+e^{-\frac{\kappa B}{20}}\|\varepsilon\|_{L^{\infty}}^2\leq b_c^{\frac{1}{2}}\mathcal{N}+b_c^{\frac{7}{2}}.
\end{align*}
For the nonlinear term, we use similar technique as before to estimate:
\begin{align*}
&\quad\;\Bigg{|}\int\psi_B\Big{[}|Q_b+\varepsilon|^{p+1}-Q_b^{p+1}-(p+1)\varepsilon Q_b^{p}-\frac{p(p+1)}{2}Q_b^{p-1}\varepsilon^2\Big{]}\Bigg{|}\\
&\lesssim\int_{y<\kappa B}\big{(}|\varepsilon|^{p+1}+Q_b^{p-2}|\varepsilon|^3\big{)}\psi_B+\int_{y>\kappa B}|\varepsilon|^{p+1}+e^{-\frac{\kappa B}{20}}\|\varepsilon\|_{L^{\infty}}^3\\
&\leq b_c^{\frac{1}{2}}\mathcal{N}+b_c^{\frac{7}{2}}.
\end{align*}
Finally, we claim there exists a constant $0<\kappa<1$ independent of $b$ (recall $\kappa$ appears in the definition of the weight function $\varphi$) such that the following holds for some universal constant $\nu_1>0$:

\begin{equation}\label{COF}
\int\big{(}\varepsilon_y^2\psi_B+\varepsilon^2\zeta_B-p\psi_B\mathcal{Q}_p^{p-1}\varepsilon^2\big{)}\geq \nu_1\mathcal{N}-\frac{1}{\nu_1}b_c^{\frac{7}{2}},
\end{equation}
Then the lower bound follows immediately. We leave the proof of \eqref{COF} in Appendix A.

This concludes the proof of Proposition \ref{MF}.
\end{proof}

\section{Existence and stability of the self-similar dynamics}
\subsection{Closing the bootstrap}In this section, we will compete the proof of Proposition \ref{BS}.

\subsubsection*{{\bf Step 1}}{\it Dynamical trapping on $b$}.

We first prove the dynamical trapping of $b$, i.e. \eqref{BS1}. Suppose for some $s_0\in[0,s^*)$, we have $\tilde{b}(s_0)\geq b_c^{\frac{3}{2}+2\nu}$. By the choice of the initial data, i.e. (2.28), we can find some $s_1\in[0,s_0)$ such that $\tilde{b}(s_1)=b_c^{\frac{3}{2}+\frac{5}{2}\nu}$ and $\tilde{b}(s)\geq b_c^{\frac{3}{2}+\frac{5}{2}\nu}$ for all $s\in[s_1,s_0)$, then $\tilde{b}_s(s_1)\geq 0$. From \eqref{BS3} and \eqref{MES3}, we have:
\begin{equation}
\tilde{b}_s(s_1)\leq-c_p\tilde{b}(s_1)b_c+b_c^{\frac{5}{2}+3\nu}\leq -c_p b_c^{\frac{5}{2}+\frac{5\nu}{2}}+b_c^{\frac{5}{2}+3\nu}<0,
\end{equation}
if $b_c$ is small enough (or equivalently $p^*(\nu)$ is close enough to 5) such that $b_c^{\nu}\ll 1$. We get a contradiction. The opposite bound is similar.

\subsubsection*{{\bf Step 2}}{\it Pointwise bound of the localised Sobolev norm of $\varepsilon$}.

The bootstrap bound \eqref{BB3} is a consequence of the monotonicity formula which we proved in the last section. We argue again by contradiction and assume that there exists $s_2\in(0,s^*)$ s.t. $\mathcal{N}(s_2)\geq b_c^{3+8\nu}$. By continuity and the choice of initial data, i.e. (2.29), we can find $s_3\in (0,s_2)$ such that for all $ s\in[s_3,s_2]$, $\mathcal{N}(s)\geq b_c^{3+10\nu}$, and $\mathcal{N}(s_3)=b_c^{3+10\nu}$. Then we have for all $ s\in[s_3,s_2]$:
$$\int\big{(}\varepsilon^2_y(s)+\varepsilon^2(s)\big{)}\varphi_B'\geq \frac{1}{B}b_c^{3+10\nu}=b_c^{3+\frac{1}{20}+10\nu}\gg b_c^{\frac{7}{2}},$$
provided that $\nu$ is chosen small enough (say $\nu=\frac{1}{1000}$).
From \eqref{MF1}, we know $d\mathcal{F}/ds\leq 0$ on $[s_3,s_2]$, which yields $\mathcal{F}(s_3)\geq\mathcal{F}(s_2)$. Thus \eqref{MF2} leads to:
\begin{multline*}
b_c^{3+8\nu}-b_c^{\frac{7}{2}}\leq \mathcal{N}(s_2)-b_c^{\frac{7}{2}}\lesssim\mathcal{F}(s_2)\leq\mathcal{F}(s_3)\lesssim \mathcal{N}(s_3)+b_c^{\frac{7}{2}}= b_c^{3+10\nu}+b_c^{\frac{7}{2}}.
\end{multline*}
This is a contradiction since $b_c^{\nu}\ll 1$. Therefore we conclude the proof of \eqref{BB3}.
\subsubsection*{{\bf Step 3}}{\it $L^{p_0}$ control of $\varepsilon$}.

For the $L^{p_0}$ norm of $\varepsilon$, it is  more convenient to work with the original variables. Consider the decomposition (see \eqref{GD}):
$$u(t,x)=Q_{S}(t,x)+\tilde{u}(t,x)=\frac{1}{\lambda(t)^{\frac{2}{p-1}}}(Q_{b(t)}+\varepsilon(t))\bigg{(}\frac{x-x(t)}{\lambda(t)}\bigg{)}.$$
By rescaling, it is sufficient to prove for all $t\in[0,T^*)$:
\begin{equation}\label{61}
\|\tilde{u}(t)\|_{L^{p_0}}\leq \frac{b_c^{\frac{13}{28}}}{\lambda(t)^{\frac{2}{p-1}-\frac{1}{p_0}}}.
\end{equation}
To prove this, we write down the equation of $\tilde{u}$ and use a refined Strichartz estimate for the Airy equations. Indeed, the equation of $\tilde{u}$ is:
$$\partial_t\tilde{u}+\tilde{u}_{xxx}=-\mathcal{E}-\big{(}f(\tilde{u})\big{)}_x$$
with
\begin{align*}
&\mathcal{E}=\frac{1}{\lambda(t)^{3+\frac{2}{p-1}}}\bigg{[}-\Phi_b+b_sP_b-\bigg{(}\frac{\lambda_s}{\lambda}+b\bigg{)}\Lambda Q_b-\bigg{(}\frac{x_s}{\lambda}-1\bigg{)}Q_b'\bigg{]}\bigg{(}t,\frac{x-x(t)}{\lambda(t)}\bigg{)},\\
&f(\tilde{u})=(Q_{S}+\tilde{u})|Q_{S}+\tilde{u}|^{p-1}-Q_{S}|Q_{S}|^{p-1},
\end{align*}
where $\Phi_b$ is defined in (3.3).

Now we state the result of D. Foschi in \cite{F} about the inhomogeneous Strichartz estimates:
\begin{proposition}[D. Foschi, Theorem 1.4 of \cite{F}]
Consider a family of linear operators $U(t)$: $H\rightarrow L_X^2$, $t\in\mathbb{R}$, where $H$ is a Hilbert space. Suppose the following properties of $U(t)$ hold:
\begin{enumerate}
\item For all $t\in \mathbb{R}$, $h\in H$:
$$\|U(t)h\|_{L_X^2}\lesssim \|h\|_{H}.$$
\item There exists a constant $\sigma>0$, such that for all $f\in L_X^1\cap L_X^2$ and $t,s\in\mathbb{R}$,  there holds:
$$\|U(t)U(s)^*f\|_{L_X^{\infty}}\lesssim \frac{1}{|t-s|^{\sigma}}\|f\|_{L_X^1}.$$
\end{enumerate}
We say a pair $(q,r)\in[2,+\infty]^2$ is $\sigma$-acceptable if and only if they satisfy:
$$\frac{1}{q}<2\sigma\bigg{(}\frac{1}{2}-\frac{1}{r}\bigg{)}\text{ or } (q,r)=(+\infty,2).$$
Consider $0<\sigma<1$ and 2 $\sigma$-acceptable pairs: $(q_i,r_i)$, $i=1,2$, such that the scaling rule is satisfied:
$$\frac{1}{q_1}+\frac{\sigma}{r_1}+\frac{1}{q_2}+\frac{\sigma}{r_2}=\sigma.$$
Then we have the following inhomogeneous Strichartz estimates:
\begin{equation}
\bigg{\|}\int_{s<t}U(t)U(s)^*F(s)ds\bigg{\|}_{L_t^{q_1}L_X^{r_1}}\lesssim\|F\|_{L_t^{q'_2}L_X^{r'_2}}.
\end{equation}
\end{proposition}

Here, we can use Proposition 6.1 to derive a refined Strichartz estimate for the Airy equations with zero initial data. Let $U(t)=\mathbf{1}_{[0,+\infty)}(t)e^{-t\partial_x^3}$, then by the theory of oscillatory integral, we have%
\footnote{See Page 13--15 in \cite{KTV}.}
:
$$
\|U(t)h\|_{L^2}\leq \|h\|_{L^2},\quad\|U(t)h\|_{L^{\infty}}\lesssim \frac{1}{|t|^{\frac{1}{3}}}\|h\|_{L^1},\quad \text{for }\forall t\not= 0.
$$
Therefore, the following refined Strichartz estimates hold for Airy equations with zero initial data:
\begin{corollary}[Refined Strichartz estimates]\label{RSEF}For all $\frac{1}{3}$-acceptable pairs $(q_1,r_1)$ and $(q_2,r_2)$, if they satisfy: $$\frac{1}{q_1}+\frac{1}{3r_1}+\frac{1}{q_2}+\frac{1}{3r_2}=\frac{1}{3},$$ then there holds:
\begin{equation}\label{RSF}
\bigg{\|}\int_0^te^{-(t-s)\partial_x^3}\big{(}h(s,\cdot)\big{)}ds\bigg{\|}_{L_t^{q_1}L_x^{r_1}}\lesssim\|h\|_{L_t^{q'_2}L_x^{r'_2}}.
\end{equation}
\end{corollary}

Now we fix $\forall t\in[0,T^*)$, and choose
$$(q_1,r_1)=(+\infty,p_0),\quad \frac{1}{r_2}=\frac{1}{p_0}-\delta,\quad \frac{1}{q_2}=\frac{p_0-2}{3p_0}+\frac{\delta}{3},$$
with $\delta>0$ to be chosen later.
It is easy to check $(q_i,r_i)$ satisfy the conditions in Corollary 6.2. Then we have the following estimate on $[0,t]$:
\begin{equation}\label{RSE}\begin{split}
\|\tilde{u}\|_{L_{[0,t]}^{\infty}L_x^{p_0}}&\lesssim\big{\|}e^{-t\partial_x^3}\big{(}\tilde{u}(0)\big{)}\big{\|}_{L_{[0,t]}^{\infty}L_x^{p_0}}+\|\mathcal{E}\|_{L^{q_2'}_{[0,t]}L_x^{r_2'}}+\big{\|}\big{(}f(\tilde{u})\big{)}_x\big{\|}_{L^{q_2'}_{[0,t]}L_x^{r_2'}}\\
&=I+II+III.
\end{split}\end{equation}
We let $\sigma_0=\frac{1}{2}-\frac{1}{p_0}(=\frac{1}{10})$, then by Sobolev embedding:
\begin{equation}\label{62}
I\lesssim\big{\|}e^{-t\partial_x^3}\big{(}\tilde{u}(0)\big{)}\big{\|}_{L_{[0,t]}^{\infty}\dot{H}^{\sigma_0}}=\frac{1}{\lambda(0)^{\frac{2}{p-1}-\frac{1}{p_0}}}\|\varepsilon(0)\|_{\dot{H}^{\sigma_0}}\leq \frac{b_c^{10}}{\lambda(t)^{\frac{2}{p-1}-\frac{1}{p_0}}}.
\end{equation}
For $II$, from \eqref{APP}, \eqref{BS1}, \eqref{BS3}, \eqref{MES1}, \eqref{MES2} and \eqref{MES3}, there holds for all $ \tau\in[0,t]$:
\begin{equation*}\begin{split}
\|\mathcal{E}(\tau)\|_{L^{r_2'}_x}&=\frac{1}{\lambda(\tau)^{2+\frac{2}{p-1}+\frac{1}{r_2}}}\bigg{\|}-\Phi_b+b_sP_b-\bigg{(}\frac{\lambda_s}{\lambda}+b\bigg{)}\Lambda Q_b-\bigg{(}\frac{x_s}{\lambda}-1\bigg{)}Q_b'\bigg{\|}_{L^{r_2'}}\\
&\lesssim \frac{1}{\lambda(\tau)^{2+\frac{2}{p-1}+\frac{1}{r_2}}}\bigg{(}\|\Phi_b\|_{L^{r_2'}}+b_c^{\frac{5}{2}}\|P_b\|_{L^{r_2'}}+\mathcal{N}^{\frac{1}{2}}+b_c^{\frac{5}{2}}\bigg{)}\\
&\lesssim \frac{b_c^{1+\frac{1}{p_0}-\delta}}{\lambda(\tau)^{2+\frac{2}{p-1}+\frac{1}{r_2}}}.
\end{split}\end{equation*}
From \eqref{48} we obtain:
\begin{equation}\label{63}
II\lesssim \Bigg{(}\int_0^t\bigg{(}\frac{b_c^{1+\frac{1}{p_0}-\delta}}{\lambda(\tau)^{2+\frac{2}{p-1}+\frac{1}{r_2}}}\bigg{)}^{q_2'}d\tau\Bigg{)}^{\frac{1}{q_2'}}\lesssim \frac{b_c^{\frac{p_0+1}{3p_0}-\frac{2\delta}{3}}}{\lambda(t)^{\frac{2}{p-1}-\frac{1}{p_0}}}= \frac{b_c^{\frac{7}{15}-\frac{2\delta}{3}}}{\lambda(t)^{\frac{2}{p-1}-\frac{1}{p_0}}}.
\end{equation}
Finally we deal with $III$. For all $\tau\in[0,t]$, there holds:
\begin{equation}\label{6001}\begin{split}
\big{\|}(f(\tilde{u}))_x\big{\|}_{L^{r_2'}}&=\frac{1}{\lambda(\tau)^{2+\frac{2}{p-1}+\frac{1}{r_2}}}\Big{\|}\big{(}(Q_b+\varepsilon)|Q_b+\varepsilon|^{p-1}-Q_b^p\big{)}_y\Big{\|}_{L^{r_2'}}\\
&\lesssim \frac{1}{\lambda(\tau)^{2+\frac{2}{p-1}+\frac{1}{r_2}}}\Big{(}\|\varepsilon_yQ_b^{p-1}\|_{L^{r_2'}}+\big{\|}\varepsilon_y|\varepsilon|^{p-1}\big{\|}_{L^{r_2'}}\\
&\qquad\qquad\qquad+\big{\|}\varepsilon(Q_b)_y|Q_b|^{p-2}\big{\|}_{L^{r_2'}}+\big{\|}(Q_b)_y|\varepsilon|^{p-1}\big{\|}_{L^{r_2'}}\Big{)}.
\end{split}\end{equation}
We estimate these terms separately. First from \eqref{BS4}, \eqref{BS5} and \eqref{CD2} we have:
\begin{align*}
&\big{\|}\varepsilon(Q_b)_y|\varepsilon|^{p-2}\big{\|}_{L^{r_2'}}\leq \|\varepsilon\|_{L^{r_2'(p-1)}}^{p-1}\leq b_c^{\frac{3}{2}},\\
&\|\varepsilon_y|\varepsilon|^{p-1}\|_{L^{r_2'}}\leq \|\varepsilon_y\|_{L^2}\|\varepsilon\|_{L^{r(p-1)}}^{p-1}\leq b_c^{\frac{3}{2}},
\end{align*}
where $$\frac{1}{r_2'}=\frac{1}{2}+\frac{1}{r}.$$
Next, by using the bootstrap bound \eqref{BS3}, \eqref{BS5} and the decay property of $Q_b$, we have:
\begin{align*}
&\quad\;\|\varepsilon_yQ_b^{p-1}\|_{L^{r_2'}}\\
&=\Bigg{(}\int_{y<-\kappa B}|\varepsilon_y|^{r_2'}Q_b^{r_2'(p-1)}+\int_{|y|<\kappa B}|\varepsilon_y|^{r_2'}Q_b^{r_2'(p-1)}+\int_{y>\kappa B}|\varepsilon_y|^{r_2'}Q_b^{r_2'(p-1)}\Bigg{)}^{\frac{1}{r_2'}}\\
&\lesssim \|\varepsilon_y\|_{L^2}\|Q_b\|_{L^{r(p-3)}}^{p-3}\|Q_b\|^2_{L^{\infty}(|y|>\kappa B)}+\|\varepsilon_y\|_{L^2({|y|}<\kappa B)}\|Q_b\|_{L^{r(p-1)}}^{p-1}\\
&\lesssim b_c^{\frac{3}{2}}.
\end{align*}
The same estimate holds for $\|\varepsilon(Q_b)_y|Q_b|^{p-2}\|_{L^{r_2'}}$.

Injecting all the above estimates into \eqref{6001} yields:
$$\big{\|}(f(\tilde{u}))_x\big{\|}_{L^{r_2'}}\lesssim \frac{b_c^{\frac{3}{2}}}{\lambda(t)^{2+\frac{2}{p-1}+\frac{1}{r_2}}}.$$
By a similar argument we have:
\begin{equation}\label{64}
III\lesssim \frac{b_c^{\frac{1}{2}}}{\lambda(t)^{\frac{2}{p-1}-\frac{1}{p_0}}}.
\end{equation}
Injecting \eqref{62}, \eqref{63} and \eqref{64} into \eqref{RSE}, we obtain \eqref{61}, provided that $\delta$ is small enough (since $\frac{1}{2}>\frac{7}{15}>\frac{13}{28}$).

This concludes the proof of Proposition \ref{BS} (Recall  we have proved \eqref{BB5} in Lemma 4.1).
\subsection{ Proof of Theorem \ref{MT}}
We are now in position to prove Theorem \ref{MT}.

Pick a $\nu>0$ small enough and a $p\in(5,p^*(\nu))$. For all $ u_0\in \mathcal{O}_p$, we choose $b^*(p)=b_c$ and denote $u(t)$ the corresponding solution to the Cauchy problem \eqref{CP} with maximal lifetime $T$. Proposition \ref{BS} implies that $u(t)$ satisfies the geometrical decomposition introduced in Section 2 on $[0,T)$:
$$u(t,x)=\frac{1}{\lambda(t)^{\frac{2}{p-1}}}(Q_{b(t)}+\varepsilon(t))\bigg{(}\frac{x-x(t)}{\lambda(t)}\bigg{)},$$
and the bounds in Proposition \ref{BS} hold on $[0,T)$. From \eqref{ECL1}, we have \eqref{11} and \eqref{12}.
\subsubsection*{{\bf Step 1}}{\it Finite time blow-up and self-similar rate}.

From \eqref{MES1} we have:
\begin{equation}\label{65}\forall t\in[0,T),\quad (1-\nu^2)b_c\leq -\lambda_t\lambda^2\leq(1+\nu^2)b_c.\end{equation}
Integrating it from $0$ to $t$ yields:
$$\forall t\in[0,T),\quad(1-\nu^2)b_ct\leq\frac{1}{3}\lambda^3(0)\text{ and hence }T\leq\frac{\lambda^3(0)}{3b_c(1-\nu^2)}<+\infty.$$
So the solution blows up in finite time. From $H^1$ Cauchy theory we have:
$$\|u_x(t)\|_{L^2}\rightarrow +\infty \text{ as } t\rightarrow T,$$ which implies $\lambda(t)\rightarrow 0$ as $t\rightarrow T$. We thus integrate \eqref{65} from $t$ to $T$ to obtain:
$$\forall t\in[0,T],\quad (1-\nu^2)b_c(T-t)\leq\frac{\lambda^3(t)}{3}\leq(1+\nu^2)b_c(T-t),$$
which implies \eqref{14}.
\subsubsection*{{\bf Step 2}}{\it Convergence of the blow-up point}.

From \eqref{MES2} we have:
$$|x_t|=\frac{1}{\lambda^2}\bigg{|}\frac{x_s}{\lambda}\bigg{|}\leq \frac{1+\nu^2}{\lambda^2}.$$
Thus from \eqref{14}, we get:
$$\int_0^T|x_t|\leq \int_0^T \frac{1+\nu^2}{\big{(}(1-\nu^2)b_c(T-t)\big{)}^{\frac{2}{3}}}\leq(1+\nu)\frac{\lambda(0)}{b_c}<+\infty,$$
and then \eqref{13} follows.
\subsubsection*{{\bf Step 3}}{\it Strong convergence in $L^q$}.

Fix a $q\in[2,\frac{2}{1-2\sigma_c})$, and let $0<\tau\ll T$ and $0<t<T-\tau$, let $u_{\tau}(t)=u(t+\tau)$ and $v_{\tau}(t')=u_{\tau}(t')-u(t')$ for all $ t'\in[t,T-\tau)$. Then $v_{\tau}$ satisfies:
$$\partial_{t'} v_{\tau}+\partial_{xxx}v_{\tau}=\big{(}u|u|^{p-1}-u_{\tau}|u_{\tau}|^{p-1}\big{)}_x.$$
Let $\sigma_1=\frac{1}{2}-\frac{1}{q}$, and chose $\tilde{q}$ and $\tilde{r}$, such that $(+\infty,q)$ and $(\tilde{q},\tilde{r})$ satisfy the conditions in Corollary \ref{RSEF}. Then we have:
\begin{equation*}\begin{split}
\big{\|}(u|u|^{p-1})_x\big{\|}_{L_x^{\tilde{r}'}}&=\frac{1}{\lambda^{2+\frac{1}{q}+\frac{1}{\tilde{r}}+\sigma_1-\sigma_c}}\Big{\|}\big{(}(Q_b+\varepsilon)|Q_b+\varepsilon|^{p-1}\big{)}_y\Big{\|}_{L^{\tilde{r}'}}\\
&\lesssim \frac{1}{\lambda^{2+\frac{1}{q}+\frac{1}{\tilde{r}}+\sigma_1-\sigma_c}}\Big{(}\big{(}\|(Q_b)_y\|_{L^2}+\|\varepsilon_y\|_{L^2}\big{)}\big{(}\|Q_b\|_{L^{r_0}}^{p-1}+\|\varepsilon\|_{L^{r_0}}^{p-1}\big{)}\Big{)}\\
&\lesssim \frac{1}{\lambda^{2+\frac{1}{q}+\frac{1}{\tilde{r}}+\sigma_1-\sigma_c}},
\end{split}\end{equation*}
where $$\frac{1}{\tilde{r}'}=\frac{1}{2}+\frac{p-1}{r_0}.$$
Since $\sigma_1<\sigma_c$ and $\lambda(t)\sim\sqrt[3]{3b_c(T-t)}$, we conclude:
\begin{align*}
&\quad\Big{\|}\big{(}u|u|^{p-1}-u_{\tau}|u_{\tau}|^{p-1}\big{)}_x\Big{\|}_{L^{\tilde{q}'}_{[t,T-\tau)}L_x^{\tilde{r}'}}\\
&\lesssim \bigg{(}\int_t^T\bigg{(}\frac{1}{\lambda(t')^{2+\frac{1}{q}+\frac{1}{\tilde{r}}+\sigma_1-\sigma_c}}\bigg{)}^{\tilde{q}'}dt'\bigg{)}^{\frac{1}{\tilde{q}'}}\\
&\lesssim \frac{1}{b_c^2}(T-t)^{\frac{\sigma_c-\sigma_1}{3}}\rightarrow 0,\text{ as } t\rightarrow T, \text{ uniformly in } \tau.
\end{align*}
\begin{remark}Here we can see the case $q=q_c$ (i.e. $\sigma_1=\sigma_c$) will lead to a logarithm on the upper bound of the critical norm, therefore the strong convergence can't exist in the critical space.
\end{remark}

Next from the refined Strichartz estimate \eqref{RSF} and Sobolev embedding we have:
\begin{equation}\label{601}
\|v_{\tau}\|_{L^{\infty}_{[t,T-\tau)}L^q_x}\lesssim \|v_{\tau}(t)\|_{\dot{H}^{\sigma_1}}+\bigg{(}\int_t^T\bigg{(}\frac{1}{\lambda(t')^{2+\frac{1}{q}+\frac{1}{\tilde{r}}+\sigma_1-\sigma_c}}\bigg{)}^{\tilde{q}'}dt'\bigg{)}^{\frac{1}{\tilde{q}'}}.
\end{equation}
We claim \eqref{601} implies that $u(t)$ is a Cauchy sequence in $L^q$ as $t\rightarrow T$. Indeed, for all $\epsilon>0$, we can choose a $t_{\epsilon}$ close enough to $T$, such that:
$$\bigg{(}\int_{t_{\epsilon}}^T\bigg{(}\frac{1}{\lambda(t')^{2+\frac{1}{q}+\frac{1}{\tilde{r}}+\sigma_1-\sigma_c}}\bigg{)}^{\tilde{q}'}dt'\bigg{)}^{\frac{1}{\tilde{q}'}}\leq \frac{\epsilon}{2C_0},$$
where $C_0$ is the implicit constant in \eqref{601}.  From $H^1$ Cauchy theory i.e. $u(t)\in C([0,T),H^1)$, there exists a $\tau_0=\tau_0(t_{\epsilon})\in(0,T-t_{\epsilon})$, such that for all $0<\tau\leq \tau_0$,
$$\|v_{\tau}(t_{\epsilon})\|_{\dot{H}^{\sigma_1}}\leq \frac{\epsilon}{2C_0}.$$
Choose a $t_0<T$ such that $T-t_0<\tau_0$. Then for all $t_1,t_2\in(t_0,T)$, $t_1<t_2$, let $\tau=t_2-t_1$. From the above discussion, we have:
$$\|u(t_2)-u(t_1)\|_{L^q}=\|v_{\tau}(t_1)\|_{L^q}\leq \|v_{\tau}\|_{L^{\infty}_{[t_{\epsilon},T-\tau)}L_x^q}\leq \epsilon,$$
which means $u(t)$ is a Cauchy sequence in $L^q$ as $t\rightarrow T$. Hence, we have proven \eqref{15}.
\subsubsection*{{\bf Step 4}}{\it Singular behavior of the asymptotic profile}.

Finally, we give the proof of \eqref{16}. Let
\begin{equation}A=b_c^{-\frac{21}{20}}, \quad R(\tau)=A\lambda(\tau) \text{ for all } \tau\in[t,T),\end{equation}
where $t$ is a fixed time close enough to $T$. Then we choose a smooth cut-off function $\chi$, with $\chi(y)=0$ if $|y|>2$, $\chi(y)=1$ if $|y|<1$. Denote
$$g(x)=\chi\bigg{(}\frac{x-x(T)}{R(t)}\bigg{)}.$$
Then by Kato's localized identity for mass, we can estimate:
\begin{align*}
&\quad\;\Bigg{|}\frac{d}{d\tau}\int u^2(\tau)g\Bigg{|}=\Bigg{|}-3\int u_x^2(\tau)g_x+\int u^2(\tau)g_{xxx}+\frac{2p}{p+1}\int |u(\tau)|^{p+1}g_x\Bigg{|}\\
&\lesssim\frac{1}{R(t)}\bigg{(}\int|u_x(\tau)|^2+|u(\tau)|^{p+1}\bigg{)}+\frac{1}{R(t)^3}\bigg{|}\int\chi'''\bigg{(}\frac{x-x(T)}{R(t)}\bigg{)}u^2(\tau)\bigg{|}\\
&\lesssim\frac{1}{R(t)}\frac{1}{\lambda(\tau)^{2-2\sigma_c}}\bigg{(}\int\big{|}(\varepsilon+Q_b)_y\big{|}^2+|\varepsilon+Q_b|^{p+1}\bigg{)}+\frac{1}{R(t)^2}\|u(\tau)\|^2_{L^{\infty}}\\
&\lesssim\frac{1}{R(t)}\frac{1}{\lambda(\tau)^{2-2\sigma_c}}+\frac{1}{R(t)^2}\frac{1}{\lambda(\tau)^{1-2\sigma_c}}\big{(}\|Q_b\|^2_{L^{\infty}}+\|\varepsilon\|^2_{L^{\infty}}\big{)}\\
&\lesssim\frac{1}{R(t)}\frac{1}{\lambda(\tau)^{2-2\sigma_c}}+\frac{1}{R(t)^2}\frac{1}{\lambda(\tau)^{1-2\sigma_c}}.
\end{align*}
Since $u(\tau)$ converges to $u^*$ in $L^2$ as $\tau\rightarrow T$, we can integrate the above inequality from $t$ to $T$ (with respect to $\tau$) and use the fact that (which follows from (4.4)):
$$\text{for } \beta<3,\quad \int_{t}^T\frac{d\tau}{\lambda(\tau)^{\beta}}\leq -2\int_t^T\frac{\lambda_t(\tau)}{b_c\lambda(\tau)^{\beta-2}}d\tau=\frac{2\lambda(t)^{3-\beta}}{b_c(3-\beta)}$$
to obtain:
\begin{equation}\label{66}\begin{split}
&\frac{1}{\lambda(t)^{2\sigma_c}}\Bigg{|}\int\chi\bigg{(}\frac{x-x(T)}{R(t)}\bigg{)}|u^*|^2-\int\chi\bigg{(}\frac{x-x(T)}{R(t)}\bigg{)}u^2(t)\Bigg{|}\\
\lesssim & \frac{1}{A\lambda(t)^{1+2\sigma_c}}\int_t^T\frac{d\tau}{\lambda(\tau)^{2-2\sigma_c}}+\frac{1}{A^2\lambda(t)^{2+2\sigma_c}}\int_t^T\frac{d\tau}{\lambda(\tau)^{1-2\sigma_c}}\\
\lesssim & \frac{1}{b_cA}=b_c^{\frac{1}{20}}.
\end{split}
\end{equation}
On the other hand we have from the geometrical decomposition \eqref{GD}:
\begin{equation}\label{67}\begin{split}
&\quad\frac{1}{\lambda(t)^{2\sigma_c}}\int\chi\bigg{(}\frac{x-x(T)}{R(t)}\bigg{)}|u(t)|^2\\
&=\int\chi\Bigg{[}\frac{1}{A}\bigg{(}y+\frac{x(t)-x(T)}{\lambda(t)}\bigg{)}\bigg{]}|Q_{b}+\varepsilon|^2dy.
\end{split}\end{equation}
From the properties of $x(t)$ and $\lambda(t)$, we know that:
$$-\frac{x(t)-x(T)}{\lambda(t)}\sim\frac{1}{b_c}\ll A.$$
Together with Lemma \ref{SSP} and \eqref{CD2} we have:
\begin{gather*}
\int\chi\Bigg{[}\frac{1}{A}\bigg{(}y+\frac{x(t)-x(T)}{\lambda(t)}\bigg{)}\bigg{]}\varepsilon^2\lesssim A\|\varepsilon\|_{L^{\infty}}^2\leq Ab_c^{\frac{149}{135}}\leq b_c^{\frac{1}{20}},\\
\int\chi\Bigg{[}\frac{1}{A}\bigg{(}y+\frac{x(t)-x(T)}{\lambda(t)}\bigg{)}\bigg{]}|Q_b|^2=\big{(}1+\delta_0(p)\big{)}\int |\mathcal{Q}_p|^2.
\end{gather*}
with $\delta_0(p)\rightarrow 0$ as $p\rightarrow 5$. Injecting these 2 estimates and \eqref{67} into \eqref{66}, yields:
\begin{align*}
\frac{1}{R(t)^{2\sigma_c}}\int\chi\bigg{(}\frac{x-x(T)}{R(t)}\bigg{)}|u^*|^2&=\frac{1}{A^{2\sigma_c}}\int|\mathcal{Q}_p|^2\big{(}1+\delta(p)\big{)}+O(b_c^{\frac{1}{40}})\\
&=\big{(}1+\delta(p)\big{)}\int|\mathcal{Q}_p|^2.
\end{align*}
with $\lim_{p\rightarrow5}\delta(p)=0$. Let $t\rightarrow T$, i.e. $R(t)\rightarrow0$, then \eqref{16} follows.

Finally, it is immediately seen from \eqref{16} that:
$$u^*\notin L^{\frac{2}{1-2\sigma_c}},$$
which concludes the proof of Theorem \ref{MT}.

\appendix
\section{Proof of \eqref{COF}.}
The coercivity result of $\mathcal{F}$ i.e. \eqref{COF}, follows from the following lemma%
\footnote{See for example, Lemma 2.1 in \cite{MMR1}.}%
:
\begin{lemma}[Coercivity of $L$]
There exists a constant $\kappa_0>0$ such that for all $f\in H^1$, there holds:
\begin{equation}\label{A1}
(Lf,f)\geq \kappa_0\|f\|_{H^1}^2-\frac{1}{\kappa_0}\Big{[}(f,\mathcal{Q}_p)^2+(f,\Lambda \mathcal{Q}_p)^2+(f,y\Lambda\mathcal{Q}_p)^2\Big{]}.
\end{equation}
\end{lemma}
Now we can prove \eqref{COF} by using Lemma A.1, orthogonality condition \eqref{OC} and a localization argument:

Choose a smooth function $\eta_0$ such that $\eta_0(y)=1$, if $y<\kappa$, $\eta_0(y)=e^{-y}$ if $y>1$ and $\eta'_0(y)\leq 0$ for all $y$. Let $$\Psi_B(y)=\psi_B(y)\eta_0(\frac{y}{B}).$$
Then we apply \eqref{A1} for $f=\varepsilon\sqrt{\Psi_B}$. We compute every term in \eqref{A1} separately:\\
First, from \eqref{APOW} and the definition of $\psi$ and $\varphi$ we have for all $y\leq \kappa B$, $$\psi_B(y)\leq (1+3\kappa)\varphi_B(y).$$
By the same strategy as in Section 5, we obtain:
\begin{align*}
(Lf,f)&=\int \varepsilon^2_y\Psi_B+\varepsilon^2\Psi_B-p\Psi_B\mathcal{Q}_p^{p-1}\varepsilon^2+\int\varepsilon^2\frac{(\Psi_B)_y^2}{4\Psi_B}-\frac{1}{2}\int \varepsilon^2(\Psi_B)_{yy}\\
&\leq \int_{y\leq \kappa B}\big{(}\varepsilon^2_y+\varepsilon^2-p\mathcal{Q}_p^{p-1}\varepsilon^2\big{)}\psi_B+O(\frac{1}{B})\int_{y<-\kappa B}\big{(}\varepsilon^2_y+\varepsilon^2\big{)}\psi_B\\
&\quad+C\int_{y>\kappa B}\big{(}\varepsilon_y^2+\varepsilon^2e^{-\frac{y}{B}}\big{)}\\
&\leq \int_{y\leq \kappa B}\big{(}\varepsilon^2_y\psi_B+\varepsilon^2\varphi_B-p\psi_B\mathcal{Q}_p^{p-1}\varepsilon^2\big{)}
+C\kappa\int_{y\leq \kappa B}\varepsilon^2\varphi_B\\
&\quad+O(\frac{1}{B})\int_{y<-\kappa B}(\varepsilon_y^2+\varepsilon^2)\varphi_B+Cb_c^{\frac{7}{2}}\\
&\leq \int\big{(}\varepsilon^2_y\psi_B+\varepsilon^2\varphi_B-p\psi_B\mathcal{Q}_p^{p-1}\varepsilon^2\big{)}+Cb_c^{\frac{7}{2}}+C(\kappa B+1)\int_{y<\kappa B}(\varepsilon_y^2+\varepsilon^2)\varphi'_B,
\end{align*}
with some constant $C>0$ independent of $\kappa$ and $B$.\\
Next, a direct computation shows:
\begin{align*}
\kappa_0\|f\|_{H^1}^2&\geq \kappa_0\int_{y\leq\kappa B} (\varepsilon^2_y\Psi_B+\varepsilon^2\Psi_B)-C\int\varepsilon^2\frac{(\Psi_B)_y^2}{4\Psi_B}\\
&\geq \frac{1}{C}\int_{y\leq \kappa B}(\varepsilon^2_y+\varepsilon^2)\psi_B-Cb_c^{\frac{7}{2}}.
\end{align*}
Then, from the orthogonality condition \eqref{OC} we have:
$$|(f,\mathcal{Q}_p)|\lesssim \int_{|y|>\kappa B}|\varepsilon|e^{-|y|}\lesssim e^{-\frac{\kappa B}{2}}\|\varepsilon\|_{L^{\infty}}\lesssim b_c^{10}.$$
The same estimates hold for $(f,\Lambda \mathcal{Q}_p)$ and $(f,y\Lambda\mathcal{Q}_p)$. Injecting all the above estimates into \eqref{A1}, we have:
\begin{equation}\label{A2}
\begin{split}
&\quad\;B\int_{y\leq \kappa B}(\varepsilon^2_y+\varepsilon^2)\varphi'_B\\
&\leq  C\int(\varepsilon^2_y\psi_B+\varepsilon^2\varphi_B-p\psi_B\mathcal{Q}_p^{p-1}\varepsilon^2 )+C(\kappa B+1)\int_{y<\kappa B}\big{(}\varepsilon^2_y+\varepsilon^2\big{)}\varphi'_B+Cb_c^{\frac{7}{2}}\\
&\leq  C\int(\varepsilon^2_y\psi_B+\varepsilon^2\varphi_B-p\psi_B\mathcal{Q}_p^{p-1}\varepsilon^2 )+\frac{B}{2}\int_{y\leq\kappa B}\big{(}\varepsilon^2_y+\varepsilon^2\big{)}\varphi'_B+Cb_c^{\frac{7}{2}},
\end{split}
\end{equation}
provided that $\kappa$ is small enough (We can take $\kappa$ such that it is independent of $b$). Then \eqref{A2} implies \eqref{COF} immediately.

\nocite{*}
\bibliographystyle{amsplain}
\bibliography{ref}
\end{document}